%% file: AutXfgShimadaRevised2.tex
\documentclass{amsart}

\usepackage{amssymb}
\usepackage{amsmath}
\usepackage{url}
\usepackage{enumerate}
\usepackage{graphicx}
\usepackage[utf8]{inputenc}
\usepackage[noend]{algpseudocode}
\usepackage{multicol}
\usepackage{datetime}
\usepackage{float}
\usepackage[normalem]{ulem}
\usepackage{multicol}
\usepackage{xcolor}
\usepackage{mathtools}

\usepackage{hyperref}
\hypersetup{
    colorlinks  =   true,
    linkcolor   =   blue,
    filecolor   =   blue,
    urlcolor    =   blue,
    citecolor   =   blue,
}

\input mymacros
\input mymacrosAut

\begin{document}

\title[Mordell-Weil groups of $K3$ surfaces]
{Mordell-Weil groups and automorphism groups of elliptic $K3$ surfaces}

\author[I. Shimada]{Ichiro Shimada}
\address{Department of Mathematics, Graduate School of Science, Hiroshima University,
1-3-1 Kagamiyama, Higashi-Hiroshima, 739-8526 JAPAN}
\email{ichiro-shimada@hiroshima-u.ac.jp}

\begin{abstract}
We present a method to calculate the action of the Mordell-Weil group of 
an elliptic  $K3$ surface 
on the numerical N\'eron-Severi lattice of the $K3$ surface.
As an application,
we compute a finite generating set of the automorphism group of a $K3$ surface
birational to the double plane branched along a $6$-cuspidal sextic curve of torus type.
\end{abstract}

\thanks{This work was supported by JSPS KAKENHI Grant Number~20H01798.}
\maketitle
\section{Introduction}
We work over an algebraically closed field $k$.
\par
Let $X$ be a projective $K3$ surface.
We denote by   $\SX$ the \emph{numerical N\'eron-Severi lattice} of $X$,
that is, the group of numerical equivalence classes of divisors of $X$ with the intersection pairing
\[
\intfnull\colon \SX \times \SX \to \Z.
\]
Let $\OG(S_X)$ denote the group of isometries of the lattice $S_X$.
We investigate the automorphism group $\Aut(X)$ of $X$
by means of the action 
\[
\Aut(X) \to \OG(S_X)
\]
of $\Aut(X)$ on the lattice $\SX$.
\par
Let $\phi\colon X\to \P^1$ be an elliptic fibration
with a distinguished section $ \zerosec \colon \P^1\to X$.
In this case, we say that $(\phi,\zerosec)$ 
is a \emph{Jacobian fibration}.
We denote by $\MW(X, \phi,  \zerosec)$ the Mordell-Weil group 
of sections of $\phi$ with  $ \zerosec$ being  the zero element.
An element $\mwsec \in \MW(X, \phi, \zerosec)$ acts on 
the generic fiber of $\phi$ by translation.
Since $X$ is minimal,
this birational automorphism of $X$ is an automorphism of $X$,
and hence we have an embedding  of $\MW(X, \phi, \zerosec)$ into $\Aut(X)$.
In this paper,
we investigate the composite homomorphism 
\begin{equation}\label{eq:MWOG}
\MW(X, \phi, \zerosec) \;\;\to\;\; \Aut(X) \;\;\to\;\; \OG(\SX).
\end{equation}
This homomorphism has been used in many situations in the study 
of automorphisms of $K3$ surfaces (see, for example,~\cite{Schuett2016}).
The purpose of this paper is to present a general algorithm to calculate~\eqref{eq:MWOG} explicitly
and to give applications.
\par
Borcherds' method~(\cite{Borcherds1, Borcherds2}) is 
a method to calculate a finite generating set of the image of
$\Aut(X)\to \OG(\SX)$ by means of  
a certain decomposition of the nef-and-big cone of $X$ into a union of polyhedral cones.
The first application of this method to the study of the automorphism group of a $K3$ surface 
was given by Kondo~\cite{Kondo1998}.
See also~\cite{Shimada2015}.
Since this method is based on lattice-theoretic computation, 
the geometric meaning of elements in the generating set obtained by this method
is not clear in general.
The homomorphism~\eqref{eq:MWOG} 
helps us to 
express the generating set geometrically.
See Remark~\ref{rem:V0}.
\par
As an application,
we calculate the automorphism group of the complex $K3$ surface $\Xfg$
obtained as the minimal resolution of the double cover
$\barX_{f,g}$ of $\P^2$ defined by
\begin{equation}\label{eq:wfg}
w^2=f(x, y, z)^2+g(x,y,z)^3,
\end{equation}
where $f$ and $g$ are very general homogeneous polynomials on $\P^2$ of degree $3$ and $2$, respectively.
Here being \emph{very general} means that 
there exist at most countably many analytic subsets of $H^0(\P^2, \OOO(3))\times H^0(\P^2, \OOO(2))$ 
with codimension $\ge 1$ such that the pair $(f, g)$ does not belong to any of them.
We prove the following:
\begin{theorem}\label{thm:genssimple}
The automorphism group $\Aut(\Xfg)$ of $\Xfg$ is generated by 
$463$ 
involutions associated with double coverings $\Xfg \to \P^2$
and $360$ elements of infinite order
in Mordell-Weil groups of Jacobian fibrations of $\Xfg$. 
\end{theorem}
Here, by a \emph{double covering},
we mean a generically finite morphism of degree $2$.
\begin{theorem}\label{thm:rats}
The automorphism group $\Aut(\Xfg)$ 
acts on the set of smooth rational curves on $\Xfg$ transitively.
\end{theorem}
The branch curve of the finite double cover $\barX_{f,g}\to \P^2$ is 
defined by the equation $f^2+g^3=0$.
This plane curve  is called a \emph{$6$-cuspidal plane sextic of torus type}, and 
was studied intensively from various points of view.
See, for example, \cite{OkaPho2002}.
In fact,  Zariski~\cite{Zariski1931} observed that there exists a $6$-cuspidal plane sextic of \emph{non}-torus type,
and  the seminal notion of \emph{Zariski pairs} emerged
from this observation.
See~\cite{Artal1994} and~\cite{ACT2008}.
In~\cite{Degtyarev2008} and~\cite{Shimada2010}, this classical example of Zariski pairs was studied 
in relation to the theory of $K3$ surfaces.
It would be an interesting problem to calculate 
the automorphism group of the $K3$ surface
obtained from the $6$-cuspidal plane sextic of non-torus type.
\par
The generating set in Theorem~\ref{thm:genssimple} is constructed in such a way that 
we can clearly see the geometric meaning of each element.
See Section~\ref{sec:AutXfg} for more precise descriptions of these automorphisms.
Remark that this generating set 
 is not minimal at all.
 \par
In fact, we give divisors of $\Xfg$ whose classes generate $S_X$.
Hence we can calculate, in principle,  the equations of the double coverings and the Jacobian fibrations
by the method given in~\cite{Shimada2014}.
The actual computation of the equations, however, would be very hard.
\par
Theorem~\ref{thm:genssimple} is proved in the following three steps.
\begin{enumerate}[(a)]
\item \label{step:find}  We find many automorphisms of $\Xfg$
 geometrically by the methods explained in Section~\ref{sec:SX}
 (especially Section~\ref{subsec:findingaut})~and Section~\ref{sec:MW}.
\item \label{step:Borcherds} We find a finite generating set of $\Aut(\Xfg)$ by 
Borcherds' method,
which will be explained in Section~\ref{sec:Borcherds}.
 \item 
We then show that the group generated by the automorphisms obtained in Step~(\ref{step:find})
 contains the generating set obtained in Step~(\ref{step:Borcherds}).
\end{enumerate}
\par
See~\cite{LieblichMaulik2018},~\cite{Looijenga2014}~and~\cite{Sterk1985} 
for general finiteness results of the automorphism group of a $K3$ surface and its action on the nef-and-big cone.
\par
This paper is organized as follows.
After fixing some notions and notation about lattices in Section~\ref{sec:notation},
we summarize in Section~\ref{sec:SX} various computational tools that are useful in the study of the geometry of $K3$ surfaces.
These tools are based on an algorithm 
 given in~\cite{Shimada2014}
 to calculate 
$\Sep(v_1,v_2)$ of \emph{separating $(-2)$-vectors} in a hyperbolic lattice.
In Section~\ref{sec:MW}, we present an algorithm to calculate the homomorphism~\eqref{eq:MWOG}.
In Section~\ref{sec:Borcherds}, we review  Borcherds' method.
We employ a graph-theoretic formulation of Borcherds' method given in~\cite[Section 4.1]{BrandhorstShimada2021}.
Sections~\ref{sec:SX}--\ref{sec:Borcherds} are intended to be summaries of computational methods in the study of $K3$ surfaces
for future reference.
In Section~\ref{sec:AutXfg}, we  calculate $\Aut(\Xfg)$ by means of all these algorithms,
and prove Theorems~\ref{thm:genssimple}~\and~\ref{thm:rats}.
We used {\tt GAP}~\cite{GAP} for the actual computation.
In the author's webpage~\cite{compdata},
we put detailed computation data about $\Aut(\Xfg)$. 
 \par
\medskip
{\bf Acknowledgement.}
This paper started as an answer to an old question by Alex Degtyarev.
A part of this work was done during the author's stay at National University of Singapore in September 2022.
Thanks are due to Professor De-Qi Zhang for his hospitality.
The author also likes to thank the anonymous referee for many valuable comments
and suggestions.
\section{Notation and terminologies}\label{sec:notation}
By a \emph{lattice}, we mean a free $\Z$-module $L$ of finite rank with a non-degenerate symmetric bilinear form 
\[
\intfnull\colon L\times L\to \Z,
\]
which we call the \emph{intersection form} (or the \emph{intersection pairing}) of $L$.
The group of isometries of a lattice $L$ is denoted by $\OG(L)$,
which we let act on $L$ from the \emph{right}.
\par
Let $L$ be a lattice.
Then the \emph{dual lattice} $L\dual$ of $L$ is defined to be
\[
\set{x\in L\tensor\Q}{\intf{x, v}\in \Z\;\;\textrm{for all}\;\; v\in L}.
\]
The finite abelian group $A(L):=L\dual/L$ is called the \emph{discriminant group} of $L$.
We say that $L$ is \emph{unimodular} if $L=L\dual$.
\par
A lattice $L$ is said to be \emph{even} if $\intf{v, v}\in 2\Z$ holds for all $v\in L$.
A \emph{root} of an even lattice $L$ is a vector $r\in L$ such that $\intf{r, r}$ is either $2$ or $-2$.
A \emph{$(-2)$-vector} of $L$ is a root $r\in L$ 
such that $\intf{r, r}=-2$.
Suppose that  $L$ is even and negative-definite.
Then the set 
\[
\Roots(L):=\set{r\in L}{\intf{r, r}=-2}
\]
is finite.
An even negative-definite lattice $L$ is called a \emph{root lattice} 
if $L$ is generated by $\Roots(L)$.
A root lattice has a basis consisting of roots
whose dual graph is a Dynkin diagram of type $\ADE$. 
See, for example,~\cite[Section 1]{Ebeling2013}
for the definition of dual graphs,  Dynkin diagrams, and their $\ADE$-types.
\par
A lattice $L$ of rank $n>1$  is said to be \emph{hyperbolic} if the signature of 
the real quadratic space $L\tensor \R$ is $(1, n-1)$.
Let $L$ be an even hyperbolic lattice.
A \emph{positive cone} of $L$ 
 is one of the two connected components of the space 
 \[
\set{x\in L\tensor\R}{\intf{x, x}>0}.
\]
Let $\PPP$ be a positive cone of $L$.
We put 
\[
\OG(L, \PPP):=\set{g\in \OG(L)}{\PPP^g=\PPP}.
\]
We have $\OG(L)=\OG(L, \PPP)\times \{\pm 1\}$.
For  
$v\in L \tensor \R$,
we put
\[
v\sperp:=\set{x\in L\tensor \R}{\intf{x, v}=0}.
\]
When $v\in \PPP\cap L$, 
the intersection $v\sperp\cap L$ is an even negative-definite sublattice of $L$, 
and hence we can effectively calculate  the finite set 
\[
\Roots(v\sperp\cap L)=\set{r\in L}{\intf{r, v}=0, \;\; \intf{r, r}=-2}
\]
of $(-2)$-vectors in $L$ perpendicular to $v$.  
\par
For $v\in L\tensor\R$ with $\intf{v, v}<0$,
we put
\[
(v)\sperp:=v\sperp\cap \PPP=\set{x\in \PPP}{\intf{x, v}=0}, 
\]
which is a real hyperplane of $\PPP$.
Let $v_1, v_2\in L\tensor\Q$ be rational vectors in $\PPP$.
Then we can calculate the finite set 
\[
\Sep(v_1, v_2):=\set{r\in L}{\;\intf{r, v_1}>0, \; \intf{r, v_2}<0, \; \intf{r, r}=-2\;}
\]
of  $(-2)$-vectors \emph{separating $v_1$ and $v_2$}. 
See~\cite{Shimada2014} for the algorithm.
As will be explained in Section~\ref{sec:SX},
this algorithm is very useful in the study of $K3$ surfaces.
\begin{definition}
By a \emph{chamber},
we mean a closed subset $D$ of $\PPP$ such that
\begin{itemize}
\item $D$ contains  a non-empty open subset of $\PPP$, and 
\item  $D$ is defined by linear inequalities $\intf{x, v_i}\ge 0$ ($i\in I$),
where $v_i$  ($i\in I$) are vectors of $L\tensor \R$ with  $\intf{v_i, v_i}<0$ 
such that the family $\set{(v_i)\sperp}{i\in I}$ of hyperplanes is locally finite in $\PPP$.
\end{itemize} 
\end{definition}
\begin{definition}
Let $D$ be a chamber.
A \emph{wall} of $D$ is a closed subset of $D$ of the form $D\cap (v)\sperp$
such that the hyperplane $(v)\sperp$ is disjoint from the interior of $D$ and that 
$D\cap (v)\sperp$ contains a non-empty open subset of $(v)\sperp$.
We say that a vector $v\in L\tensor \R$ \emph{defines} a wall $w$ of $D$ if 
$w=D\cap (v)\sperp$ and  $\intf{x, v}>0$ for an interior point $x$ of $D$
(and hence $\intf{x, v}\ge 0$ for all $x\in D$).
A defining vector of a wall of a chamber is unique up to positive multiplicative constant.
\end{definition}
\begin{definition}
Let $\FFF:=\set{(v_{\alpha})\sperp}{\alpha\in F}$ be a locally finite family of hyperplanes in $\PPP$.
Then the closure in $\PPP$ of each connected component of
\[
\PPP\;\setminus\; \bigcup_{\alpha\in F} \;(v_{\alpha})\sperp
\]
is a chamber.
Let $\CCC_{\FFF}$ be the set of these chambers.
In this situation, we say that $\PPP$ is \emph{tessellated by the chambers in $\CCC_{\FFF}$}.
If a subset $N$ of $\PPP$ is the union of chambers in a subset  of $\CCC_{\FFF}$,
we say that $N$ is \emph{tessellated by chambers in $\CCC_{\FFF}$}.
\par
Let $w$ be a wall of a chamber $D\in \CCC_{\FFF}$.
Then there exists a unique chamber $D\sprime\in \CCC_{\FFF}$ 
such that $D\ne D\sprime$ and $w\subset D\sprime$.
This chamber $D\sprime$ is called the chamber \emph{adjacent to $D$ across the wall $w$}.
\end{definition}
A $(-2)$-vector $r\in L$ defines a reflection 
\[
s_r\colon x\mapsto x+\intf{x, r} r
\]
into the mirror $(r)\sperp$.
We have $s_r\in \OG(L, \PPP)$.
Let $W(L)$ denote the subgroup of $\OG(L, \PPP)$ generated by all the reflections $s_r$ with respect 
to $(-2)$-vectors $r$.
We call $W(L)$ the \emph{Weyl group} of $L$.
Note that the family of hyperplanes $(r)\sperp$ defined by $(-2)$-vectors $r$
%
%
is locally finite in $\PPP$.
\begin{definition}
A \emph{standard fundamental domain of  $W(L)$} is 
the closure of a connected component of 
\[
\PPP\;\setminus\; \bigcup \;(r)\sperp,
\]
where 
$r$ runs through the set of $(-2)$-vectors.
\end{definition}
%
Let  $D$ be a standard fundamental domain of  $W(L)$.
We put 
\[
\OG(L, D):=\set{g\in \OG(L)}{D^g=D}.
\]
Then we have $\OG(L, \PPP)=W(L)\semidirectproduct \OG(L, D)$.
The action of $\OG(L, \PPP)$ on $\PPP$
preserves the tessellation of $\PPP$ by 
the standard fundamental domains of $W(L)$.
\section{The numerical N\'eron-Severi lattice of a $K3$ surface}\label{sec:SX}
Let $X$ be a $K3$ surface, and
$S_X$ the lattice of numerical equivalence classes of 
divisors of $X$,
which we call the \emph{numerical N\'eron-Severi lattice} of $X$.
For a divisor $D$ of $X$,  we denote by $[D]\in S_X$ the class of $D$.
Suppose that $S_X$ is of rank $n >1$.
Then  $S_X$ is an even hyperbolic lattice.
Let $\PPP_X$ be the positive cone
of $S_X$ containing an ample class 
of $X$,
and $\barPPP_X$ the closure of $\PPP_X$ in $S_X\tensor\R$.
We put 
\begin{eqnarray*}
N_X &:=&\set{x\in \PPP_X}{\intf{x, [C]}\ge 0\;\;\textrm{for all curves $C$ on $X$}}, \\
N_X\sp{\circ}&:=&\textrm{the interior of $N_X$}, \\
\clN_X &:=&\textrm{the closure of $N_X$ in $\barPPP_X$}.
\end{eqnarray*}
The cone  $N_X$ is called the \emph{nef-and-big cone} of $X$.
If $C$ is a smooth rational curve on $X$,
then its class $[C]$ is a $(-2)$-vector of $S_X$.
We put
\[
\Rats(X):=\set{[C]\in S_X}{\textrm{$C$ is a smooth rational curve on $X$}}.
\]
We have the following:
\begin{theorem}\label{thm:NX}
The nef-and-big cone $N_X$ is a standard fundamental domain of 
the Weyl group $W(S_X)$ of $S_X$.
A $(-2)$-vector $r\in S_X$
belongs to  $\Rats(X)$ if and only if $r$ defines a wall of the chamber $N_X$.
\qed
\end{theorem}
Suppose that we have an ample class $\ampleX\in N_X\spcirc \cap S_X$.
Then Vinberg's algorithm~\cite{VinbergBombay}
enables us to  enumerate, for a given positive integer $m$, 
all the walls $N_X\cap (r)\sperp$ of $N_X$ defined by $r\in \Rats(X)$ with $\intf{r, \ampleX}\le m$.
(See~\eqref{eq:lassicalVinberg} below.)
Our algorithm~\cite{Shimada2014} of 
calculating the set $\Sep(v_1, v_2)$ of separating $(-2)$-vectors
provides us with an alternative method to investigate the nef-and-big cone $N_X$.
Below are some examples.
\subsection{Finding an ample class}\label{subsec:findanample}
It is well-known that 
a class $v\in S_X$ is ample if and only if $v\in N_X\sp{\circ}$.
Let $\clX$ be a normal surface birational to $X$,
and $h\in S_X$ the pull-back of an ample class of $\clX$
by the minimal resolution $X\to \clX$.
Then we have $h\in N_X$.
It is known~\cite{Artin1966} that $\clX$ has only rational double points 
as its singularities, 
and hence the exceptional locus of the desingularization $X\to \clX$
is a union of smooth rational curves whose dual graph is 
a Dynkin diagram of type $\ADE$.
Let $r_1, \dots, r_{\mu}$ be the classes of smooth rational curves contracted  by 
$X\to \clX$.
Then, \emph{locally around $h$}, the chamber $N_X$ is defined by $\intf{x, r_i}\ge 0$ for $i=1, \dots, \mu$.
Therefore a vector $v\in \PPP_X\cap S_X$ is ample 
if and only if 
\[
\Sep(h, v)=\emptyset , \quad \Roots(\,v\sperp\cap S_X\,)=\emptyset, 
\quad\textrm{and}\quad  \intf{ v, r_i}> 0\;\;\textrm{for}\;\; i=1, \dots, \mu.
\]
If $a\sprime\in S_X$ satisfies $ \intf{a\sprime, r_i}> 0$ for $i=1, \dots, \mu$,
then $ \ampleX:=m h +a\sprime$ is ample for sufficiently large  integers $m$.
\subsection{Nefness and ampleness}
Suppose that we have an ample class $\ampleX\in S_X$.
Then 
we can characterize $N_X$ as the unique standard fundamental domain of $W(S_X)$
containing $\ampleX$.
Let $v\in S_X$ be a vector with $\intf{v,v}>0$.
Then we have 
\[
v \in \PPP_X \;\; \Longleftrightarrow\;\; \intf{\ampleX, v}>0.
\]
When  these are the case, we have  
\[
v \in N_X  \;\; \Longleftrightarrow\;\;  \Sep(\ampleX, v)=\emptyset.
\]
When  these are the case, we have  
\[
v \in N_X\sp{\circ}   \;\; \Longleftrightarrow\;\;\Roots(\,v\sperp\cap S_X\,)=\emptyset.
\]
\subsection{The group $\OG(S_X, N_X )$}\label{subsec:OSXNX}
Recall that $\OG(S_X, N_X )$ is the subgroup of $\OG(S_X, \PPP_X)$ 
consisting of all isometries $g$ such that $N_X^g=N_X$.
Suppose again that we have an ample class $\ampleX\in S_X$.
Let $g$ be an element of $\OG(S_X)$.
Then we have 
\[
g \in \OG(S_X, \PPP_X )\;\; \Longleftrightarrow\;\; \intf{\ampleX,  \ampleX^g}>0.
\]
When  these are the case, we have  
\begin{equation}\label{eq:Sepaag}
g \in \OG(S_X, N_X ) \;\; \Longleftrightarrow\;\;  \Sep( \ampleX, \ampleX^g)=\emptyset,
\end{equation}
because, for $g\in \OG(S_X, \PPP_X)$,
the chamber $N_X^g$ is also a standard fundamental domain of $W(S_X)$.
\subsection{The set  $\Rats(X)$}\label{subsec:Rats}
Again we assume that we have an ample class $\ampleX\in S_X$.
Let $r\in S_X$ be a  $(-2)$-vector  such that $\intf{\ampleX, r}>0$.
Then there exists an effective divisor $D$ of $X$ such that $r=[D]$.
We have $r\in \Rats(X)$ if and only if $D$ is irreducible.
 \par
 Since $D$ contains a smooth rational curve $C$ such that $\intf{[C], [D]}<0$
 as an irreducible component, we have the following criterion,
 which is a geometric interpretation of Vinberg's algorithm~\cite{VinbergBombay} 
 applied to $(-2)$-vectors:
\begin{equation}\label{eq:lassicalVinberg}
 r\in \Rats(X)
 \;\; \Longleftrightarrow\;\;
 \textrm{$\intf{r, r\sprime}\ge 0$ for all $r\sprime \in \Rats(X)$ with $\intf{r\sprime, \ampleX}<\intf{r, \ampleX}$}
\end{equation}
Thanks to the algorithm to calculate $\Sep(v_1, v_2)$,
we obtain another criterion.
 \begin{proposition}\label{prop:Rats}
 Let $r\in S_X$ be a $(-2)$-vector with $\intf{\ampleX, r}>0$.
 We put
 \[
 a_r\sprime:=\ampleX+\frac{\intf{\ampleX, r}}{2} r.
 \]
 Then $r\in \Rats(X)$ if and only if 
\begin{equation}\label{eq:ar}
 \Roots(\, a_r\sp{\prime\perp} \cap S_X\,)=\{r, -r\}
 \quand 
 \Sep( a_r\sprime, \ampleX)=\emptyset.
\end{equation}
 \end{proposition}
 \begin{proof}
Since $\intf{a_r\sprime, r}=0$ and $\intf{a_r\sprime, a_r\sprime}>0$,
we have  $a_r\sprime \in (r)\sperp\subset \PPP_X$, and 
hence the set  $\Sep( a_r\sprime, \ampleX)$ makes sense.
In fact, the point 
 $a_r\sprime \in (r)\sperp$ is the image of $\ampleX$  by the orthogonal projection 
 to the hyperplane $(r)\sperp$ in $\PPP$.
 In particular, we have $\{r, -r\} \subset  \Roots(\,a_r\sp{\prime\perp} \cap S_X\,)$.
Then  Proposition~\ref{prop:Rats} follows from~\cite[Proposition 2.2]{VinbergSurvey}.
We present proof for the convenience of readers.
 \par
 If~\eqref{eq:ar} holds, then $a_r\sprime \in N_X$ and 
 a small neighborhood of $a_r\sprime$ in $(r)\sperp$ 
 is contained in $N_X$.
 In particular,  $r$ is a defining $(-2)$-vector of a wall of $N_X$ and hence $r\in \Rats(X)$.
 Conversely, suppose that $r\in \Rats(X)$.
 Then for any $r\sprime\in \Rats(X)$ with  $r\sprime\ne r$,
we have  
 $\intf{r, r\sprime}\ge 0$ and $\intf{\ampleX, r\sprime}>0$, and hence 
 \[
 \intf{a_r\sprime, r\sprime}=\intf{\ampleX, r\sprime}+\frac{\intf{\ampleX, r}\intf{r, r\sprime}}{2}\;\;>\;\;0.
 \]
 Therefore~\eqref{eq:ar} holds.
 \end{proof}
\subsection{Nefness of a vector of norm $0$}\label{subsec:norm0}
Suppose again that we have $\ampleX\in N_X\sp{\circ}\cap S_X$.
\begin{proposition}\label{prop:fcriterion}
Let $f$ be a non-zero vector in $ \barPPP_X\cap S_X$
with $\intf{f, f}=0$.
Then $f\in \clN_X$ if and only if \,$\Sep(a\sprime_f, \ampleX)=\emptyset$,
where $a\sprime_f:=\ampleX+\intf{\ampleX, f} f$.
\end{proposition}
\begin{proof}
First note that, since $f\in  \barPPP_X\setminus\{ 0\}$,
we have $\intf{\ampleX, f}>0$, $a\sprime_f\in \PPP_X$, and hence 
$\Sep(a\sprime_f, \ampleX)$ makes sense.
\par
Suppose that $f\in \clN_X$.
Since $\ampleX\in N_X\sp{\circ}$, we have 
 $a\sprime_f \in N_X\sp{\circ}$ and hence $\Sep(a\sprime_f, \ampleX)=\emptyset$.
 Suppose that $f\notin \clN_X$.
 Then there exists a smooth rational curve $C$  such that $\intf{f, [C]}<0$.
 We put $r:=[C]$. 
 Then we have $\intf{f, r}\le -1$.
 Since $\intf{f, f}=0$ and $\intf{f, \ampleX}>0$,
 there exists an effective divisor $F$ on $X$ such that $f=[F]$.
 Then $C$ is an irreducible component of $F$ such that $C\ne F$,
 and hence $\intf{\ampleX, r}< \intf{\ampleX, f}$.
 The intersection point of $(r)\sperp$ and the open line segment 
 \[
(\ampleX, f):= \set{p(t)=\ampleX +tf}{t\in \R_{> 0}}\;\;\subset\;\; \PPP_X
 \]
 is equal to $p(t_0)$,
 where 
 \[
 t_0:=-\frac{\intf{\ampleX, r}}{\intf{f, r}}\le \intf{\ampleX, r}<\intf{\ampleX, f}.
 \]
 Since $a\sprime_f=p(\intf{\ampleX, f})$,  
 the intersection point $p(t_0)$ is located 
 on the open line segment $(\ampleX, a\sprime_f) \subset (\ampleX, f)$.
 Therefore $r$ is a $(-2)$-vector separating $a\sprime_f$ and $\ampleX$.
\end{proof}
\subsection{Singularities of a normal surface birational to $X$}\label{subsec:calculatingSing}
Suppose again that we have $\ampleX\in N_X\sp{\circ}\cap S_X$.
Let $h$ be a vector  in  
$ N_X \cap S_X$, 
and let $\LLL$ be a line bundle whose class is $h$.
Then, for some large positive integer $m$,
the complete linear system $|\LLL\sp{\tensor m}|$ gives a birational morphism $X\to \barX$
to a normal surface $\barX$.
See~Saint-Donat~\cite{SaintDonat1974}.
The surface $\barX$ is smooth if and only if $h\in N_X\sp{\circ}$.
Suppose that $h\notin N_X\sp{\circ}$.
Then the singularities of $\barX$ consist of rational double points (see Artin~\cite{Artin1966}),
and the set of classes of smooth rational curves  contracted by 
 the birational morphism  $X\to \barX$ is equal to 
 \[
 \set{r\in \Rats(X)}{\intf{r, h}=0}\;\;=\;\;\Rats(X)\cap \Roots(h\sperp \cap S_X).
 \]
\subsection{Finding automorphisms from nef vectors of norm $2$}\label{subsec:findingaut}
Let $\ampleX\in S_X$ be an ample class of $X$.
Let $h$ be a vector  in  
$ N_X \cap S_X $
with  $\intf{h, h}=2$.
By a \emph{double covering},
we mean a generically finite morphism of degree $2$.
By abuse of notation, 
we write $|h|$ for the complete linear system of a line bundle whose class is $h$.
Then either one of the following holds
(see~Saint-Donat~\cite{SaintDonat1974} or Nikulin~\cite{Nikulin1990}).
\begin{itemize}
\item The complete linear system $|h|$ is base-point free
and defines a double covering   $\pi(h)\colon X\to \P^2$, or
\item  $|h|$ has a fixed component $Z$,
which is a smooth rational curve, 
and every member of $|h|$  is of the form $Z+E_1+E_2$,
where $E_1$ and $E_2$ are members of a pencil $|E|$ of elliptic curves
such that $\intf{[E], [Z]}=1$.
\end{itemize}
These two cases can be distinguished by the following criterion. 
We put
\[
\EEE:=\set{e\in S_X}{\intf{e, e}=0, \intf{e, h}=1}.
\]
Since the quadratic part of the intersection form $\intf{\phaa, \phaa}$  
restricted to the affine hyperplane of $S_X\tensor\R$ 
defined by $\intf{x, h}=1$ is 
negative-definite,
the set $\EEE$ is finite and can be calculated effectively.
\begin{itemize}
\item If $\EEE=\emptyset$, then $|h|$ is base-point free.
In this case,  we say that $h$ is a \emph{polarization of degree $2$}, and 
denote by 
$\invol(h)\in \Aut(X)$ the involution associated with the double covering 
$\pi(h)\colon X\to \P^2$ given by $|h|$.
Let 
\[
X\;\to\;\clX \;\to\; \P^2
\]
 be the Stein factorization of $\pi(h)$,
and let $B(h)\subset \P^2$ be the branch curve of the finite double covering $\clX\to \P^2$.
We can calculate the set 
\[
\Rats(X)\cap \Roots(h\sperp\cap S_X)
\]
of classes of smooth rational curves contracted by $\pi(h)$.
Hence we obtain the $\ADE$-type of $\Sing(B(h))$,
and the invariant part 
\[
\set{v\in S_X\tensor\Q}{v^{\invol(h)}=v}
\]
of  the action of $\invol(h)$ on $S_X\tensor\Q$.
Indeed, applying  to $\clX$ the theory of  \emph{canonical resolutions}
of rational double points due to Horikawa~\cite{Horikawa1975},
we have a successive blowing up $Y\to \P^2$
of $\P^2$ such that $X\to \P^2$ factors through a finite double covering $X\to Y$,
and the invariant part is equal to the pull-back of  the space 
$S_Y\tensor\Q$ of the numerical equivalence classes of curves on the rational surface $Y$.
See~\cite{Shimada2016} for detail.
From this subspace, we can calculate 
 the action of the involution $\invol(h)$ on $S_X$,
 because $\invol(h)$  acts on the orthogonal complement of the invariant subspace 
 as the scalar multiplication by $-1$.
\begin{remark}
The equality $\invol(h)=\invol(h\sprime)$ of involutions does not imply $h=h\sprime$ in general.
See Remark~\ref{rem:IJandJI}, for example.
The set of polarizations $h$ of degree $2$  that induce the same involution $\invol(h)$ is in one-to-one correspondence with 
the set of blowing-downs of $Y$ to $\P^2$.
\end{remark}
\item Suppose that $\EEE\ne \emptyset$.
Then we have a unique element $f\in \EEE$ such that
\[
f\in \clN_X
\quand z:=h-2f \in \Rats(X).
\]
We can find this $f$ by the methods in Sections~\ref{subsec:norm0} ~and~\ref{subsec:Rats}.
Then $f$ is the class of a fiber of a Jacobian fibration $\phi\colon X\to \P^1$
with $z$ being the class of the zero section $\zerosec\colon \P^1\to X$.
 From these vectors $f, z$, 
we can calculate  the Mordell-Weil group $\MW(X, \phi, \zerosec)$ 
and its action on $S_X$ by the algorithm explained in Section~\ref{sec:MW}.
\end{itemize}
%
%
%
\section{The action of a Mordell-Weil group on $S_X$}\label{sec:MW}
In this section, we assume that the characteristic of the base field $k$ is $\ne 2, 3$ for simplicity.
Let $X$ be a $K3$ surface, and $\ampleX\in S_X$ an ample class.
\par
Let $\phi\colon X\to \P^1$  be a fibration 
whose general fiber is a curve of genus $1$.
Suppose that $\phi$ has  a distinguished section  $\zerosec\colon \P^1\to X$,
that is, the pair $(\phi, \zerosec)$ is a \emph{Jacobian fibration}.
Let $\eta=\Spec k(\P^1)$ be the generic point of the base curve $\P^1$.
Then the generic fiber $E_{\eta}:=\phi\inv(\eta)$ of $\phi$ is 
an elliptic curve defined over $k(\P^1)$ with the zero element being the $k(\P^1)$-rational point
corresponding to 
$\zerosec$,
and the set 
\[
\MWphi:=\MW(X, \phi, \zerosec)
\]
 of sections of $\phi$ 
has a structure of the abelian group
with $\zerosec=0$.
This group $\MWphi $ is called the \emph{Mordell-Weil group}.
The group $\MWphi $ acts on $E_{\eta}$ 
via  the translation $x\mapsto x \addE \;\sigma$ on $E_{\eta}$,
where $\sigma\in \MWphi $ is a section and 
$\addE$ denotes the addition in the elliptic curve $E_{\eta}$.
Since $X$ is minimal, this automorphism of $E_{\eta}$
gives an automorphism of $X$.
Hence $\MWphi $ embeds in $\Aut(X)$,
and acts on the lattice $S_X$:
\begin{equation}\label{eq:MWrep}
\MWphi \to \Aut(X) \to \OG(S_X, \PPP_X).
\end{equation}
Let $f\in S_X$ be the class of a fiber of $\phi$,
and $z=[\zerosec]\in S_X$  the class of the image of  $\zerosec$.
Since the Jacobian fibration $(\phi, \zerosec)$ is uniquely determined by
the classes $f$ and $z$, we sometimes write $\MW(X, f, z)$ for $\MW(X, \phi, \zerosec)$.
The purpose of this section is to show that
we can calculate 
 the homomorphism~\eqref{eq:MWrep}
 from the classes $f$, $z$ and an ample class $\ampleX$.
 \par
 We review the theory of elliptic $K3$ surfaces, and fix some notation.
 Since $\intf{f,f}=0$, $\intf{f,z}=1$ and $\intf{z,z}=-2$,
 the classes $f$ and $z$ generate a unimodular hyperbolic sublattice $U_{\phi}$ in $S_X$ of rank $2$.
 Let $W_{\phi}$ denote the orthogonal complement of $U_{\phi}$ in $S_X$.
 Since $U_{\phi}$ is unimodular,
 we have an orthogonal direct-sum decomposition 
 \[
 S_X=U_{\phi}\oplus W_{\phi}.
 \]
 Since $W_{\phi}$ is negative-definite,
 we can calculate 
 the set 
 \[
 \Roots (W_{\phi})=\set{r\in W_{\phi}}{\intf{r,r}=-2}.
 \]
Hence  we can compute 
 \begin{equation}\label{eq:Theta}
 \Theta_{\phi}:=\Roots (W_{\phi})\cap \Rats(X)
\end{equation}
by Proposition~\ref{prop:Rats}.
 Let $\Sigma_{\phi}$ denote the sublattice of $W_{\phi}$ generated by $ \Roots (W_{\phi})$,
 and  $\tau_{\phi}$  the $\ADE$-type of the root lattice $\Sigma_{\phi}$.
Here an \emph{$\ADE$-type}  is a finite formal sum of the symbols $A_{\ell}$, $D_{\ell}$,  and $E_{\ell}$.
See, for example,~\cite[Section 1]{Ebeling2013} for the definition of $\ADE$-types of root lattices.
Then we have the following proposition.
The first part  follows from the definition of $\Rats(X)$,
and  the second part  follows from the classification of singular 
fibers of elliptic surfaces due to Kodaira and N\'eron.
See~\cite[Chapters 5 and 6]{MWLbook}. 
 \begin{proposition}\label{prop:ThetaSigma}
 The set $\Theta_{\phi}$ defined by~\eqref{eq:Theta} 
 is equal to the set of classes of smooth rational curves that are 
 contracted to points by $\phi$ and are disjoint from the zero section $\zerosec$.
 The vectors in $\Theta_{\phi}$ form a basis of the root lattice $\Sigma_{\phi}$, and 
their dual graph 
 is the Dynkin diagram of type $\tau_{\phi}$.
 \qed
 \end{proposition}
\begin{definition}
The sublattice $U_{\phi}\oplus \Sigma_{\phi}$ of $S_X$ 
is called the \emph{trivial sublattice} of the Jacobian fibration $(\phi, \zerosec)$.
\end{definition}
The following is of fundamental importance in the theory of Mordell-Weil groups.
This holds, not only for $K3$ surfaces,  but also for elliptic surfaces in general.
See~\cite[Chapter 6]{MWLbook}.
\begin{theorem}\label{thm:ShidaTate}
Let $\;[\phantom{a}]\colon \MWphi \to \Rats(X)$ denote the mapping that
associates to each section $\mwsec\in \MWphi $ the class $[\mwsec]\in  \Rats(X)$ of the image 
of $\mwsec$.
Then the composite
\begin{equation}\label{eq:MWisom}
 \MWphi \;\;\maprightsp{\hbox{\tiny{$[\phantom{a}]$}}}\;\;   \Rats(X)\;\;\inj\;\;
 S_X \;\;\surj\;\;  S_X/(U_{\phi}\oplus \Sigma_{\phi})
\end{equation}
is an isomorphism of abelian groups.
\qed
\end{theorem}
\begin{remark}\label{rem:MWL}
By the isomorphism~\eqref{eq:MWisom},
Shioda~\cite{ShiodaMWL} (see also~\cite{MWLbook}) 
introduced a structure of the positive-definite lattice 
(with a \emph{$\Q$-valued} intersection form)
on the free $\Z$-module $\MWphi/(\mathrm{torsion})$.
This lattice is called the \emph{Mordell-Weil lattice}.
The norm of the Mordell-Weil lattice is very useful,
for example, in finding good generators of $\MWphi$.
See Section~\ref{subsec:orbb6}.
\end{remark}
For a vector $v\in S_X$,
we denote by $s(v)\in \MWphi $  the section that corresponds 
to $v \bmod (U_{\phi}\oplus \Sigma_{\phi})$ by the isomorphism~\eqref{eq:MWisom}.
First, we will explain a method to calculate $[s(v)]\in \Rats(X)$ for a given $v\in S_X$.
\par
We review 
the Kodaira--N\'eron theory of singular 
fibers of an elliptic surface in more detail.
See~\cite[Chapters 5 and 6]{MWLbook},~\cite{Kodaira1963},~\cite{Neron1964}, and~\cite[Table in page 46]{Tate1975}.
Recall that
$\Theta_{\phi}$ 
is the set of classes of smooth rational curves in fibers of $\phi$ 
that is disjoint from the zero section $\zeta$, and 
 that the dual graph of 
$\Theta_{\phi}$ is the Dynkin diagram of type $\tau_{\phi}$.
 Let 
\begin{equation}\label{eq:Thetas}
 \Theta_{\phi}=\Theta_1\sqcup \cdots \sqcup \Theta_n
\end{equation}
 be the decomposition 
 according to the decomposition of the Dynkin diagram
 into connected components.
 Then two elements $r=[C]$ and $r\sprime=[C\sprime]$ of $\Theta_{\phi}$,
 where $C$ and $C\sprime$ are smooth rational curves on $X$,
 belong to the same $\Theta_{\nu}$ 
 if and only if $\phi$ maps $C$ and $C\sprime$ to the same point.
 Hence the set 
 $\{\Theta_1, \dots, \Theta_n\}$  
is in one-to-one correspondence with  the set 
 \[
\set{p\in \P^1}{\textrm{$\phi\inv(p)$ is reducible}}=\{p_1, \dots, p_n\}
 \]
in such a way that  $p_{\nu}\in \P^1$ is the point $\phi(C)$ for  $[C]\in \Theta_{\nu}$.
 We put
 \[
 \rho(\nu):=\Card(\Theta_{\nu}), 
 \quad 
\tau_{\nu}:=\textrm{the $\ADE$-type of $\Theta_{\nu}$}.
 \]
 In particular, each $\tau_{\nu}$ is either $A_{\ell}$, $D_{\ell}$, or $E_{\ell}$, and we have
$ \tau_{\phi}=\tau_1+\dots+\tau_n$.
Recall that $\Sigma_{\phi}$ is the root lattice generated by $\Theta_{\phi}$.
 Let $\Sigma_{\nu}$ be the sublattice of $\Sigma_{\phi}$ generated by the elements of $\Theta_{\nu}$.
 We have an orthogonal direct-sum decomposition
 \[
 \Sigma_{\phi}=\Sigma_{1}\oplus\cdots\oplus \Sigma_{n}.
 \]
The fiber $\phi\inv(p_{\nu})$ consists of $\rho(\nu)+1$ smooth rational curves 
\[
 C_{\nu, 0},  \; C_{\nu,1}, \; \dots,  \; C_{\nu,\rho(\nu)}
\]
 such that $\Theta_{\nu}=\{[C_{\nu,1}], \dots,  [C_{\nu,\rho(\nu)}]\}$ and that
 $ C_{\nu, 0}$ intersects the zero section  $\zerosec$.
The dual graph  of 
 \[
 \tilTheta_{\nu}:=\{[C_{\nu, 0}]\} \cup \Theta_{\nu}
 \]
 is the \emph{affine} Dynkin diagram of type $\tau_{\nu}$.
 We number the smooth rational curves in  $ \tilTheta_{\nu}$  as in Figure~\ref{fig:affDynkin}.
 \input figures.tex
 The divisor $\phi\sp{*}(p_{\nu})$ is written as 
 \[
 \phi\sp{*}(p_{\nu})=\sum_{j=0}^{\rho(\nu)} m_{\nu, j} C_{\nu, j}
 \qquad (m_{\nu, j}\in \Z_{>0}),
 \]
where the coefficients $m_{\nu, j} $ are given in Table~\ref{table:ms}. 
We put
 \[
 J_{\nu}:=\set{j}{m_{\nu, j} =1}.
 \]
 We have $0\in J_{\nu}$, and 
the class $[C_{\nu, 0}]$ is calculated by 
\begin{equation}\label{eq:Cnu0}
[C_{\nu, 0}]=f-\sum_{j=1}^{\rho(\nu)} m_{\nu, j} [C_{\nu, j}].
\end{equation}
(It is well known that $m_{\nu, j}$ with $j>0$ are the coefficients of the highest root of the root system $\Theta_{\nu}$.)
\begin{table}
\renewcommand{\arraystretch}{1.2}
\[
\begin{array}{lll}
\tau_{\nu} & j= &0,1, 2,\dots, \rho(\nu) \\
\hline 
A_{\ell} && 1,1,1, \dots, 1,1\\
D_{\ell} && 1,1,1,2,\dots, 2,1\\
E_6&& 1, 2, 1, 2, 3, 2, 1 \\
E_7 && 1, 2, 2, 3, 4, 3, 2, 1\\
E_8&&1, 3, 2, 4, 6, 5, 4, 3, 2 \\
\end{array}
\]
\vskip .3cm
\caption{Coefficients $m_{\nu, j} $}\label{table:ms}
\end{table}
Let $ \phi\sp{*}(p_{\nu})^{\sharp}$ denote the smooth part of the divisor $\phi\sp{*}(p_{\nu})$:
 \[
 \phi\sp{*}(p_{\nu})^{\sharp}=\bigcup _{j\in J_{\nu}} C_{\nu, j}\sp{\circ},
 \]
 where  $C_{\nu, j}\sp{\circ}$ is $C_{\nu, j}$ minus the intersection points of  $C_{\nu, j}$ with 
other irreducible components of $\phi\inv(p_{\nu})$. 
By Kodaira--N\'eron theory, 
we can equip   $\phi\sp{*}(p_{\nu})^{\sharp}$ with the structure of an abelian Lie group.
See~\cite[Section 5.6.1]{MWLbook}.
(When we work over $\C$,
this group structure is obtained as  the limit of the group structures of general fibers of $\phi$.)
 Then 
 the set $J_{\nu}$,
 which is regarded as the set of connected components $C_{\nu, j}\sp{\circ}$ of  $\phi\sp{*}(p_{\nu})^{\sharp}$, 
 also has a natural structure of an abelian group
 as a quotient group of $\phi\sp{*}(p_{\nu})^{\sharp}$.
The element $0\in J_{\nu}$ is the zero element.
See~Table~\ref{table:J}, which is copied from~\cite[Table in page 46]{Tate1975},  for the precise description of 
 the group structure of $J_{\nu}$.
  \begin{table}
 \[
 \renewcommand{\arraystretch}{1.4}
 \begin{array} {lll}
 \tau_{\nu} & J_{\nu} & \textrm{Group structure}\\
 \hline
A_{\ell} &\{0,1, \dots, \ell\}  & \textrm{cyclic group $\Z/(\ell+1)\Z$: the sum of $a, b\in J_{\nu}$ is}\\
 & &  \textrm{$c\in  J_{\nu}$ such that $a+b\equiv c \bmod (\ell+1)$} \\
D_{\ell}\;\; (\ell: \textrm{even})&\{0,1, 2, \ell\}   &\Z/2\Z \times \Z/2\Z  \\
D_{\ell}\;\; (\ell: \textrm{odd})&\{0,1, 2, \ell\}   &
\textrm{$\Z/4\Z$  generated by $1\in J_{\nu}$ 
with $\ell\in J_{\nu}$ being of order $2$} \\
E_6 & \{0,2,6\}  & \Z/3\Z \\
E_7 & \{0,7\}  & \Z/2\Z \\
E_8 & \{0\}  & \textrm{trivial}
 \end{array}
 \]
 \vskip 10pt 
 \caption{Group structure of $J_{\nu}$ (\cite[Table in page 46]{Tate1975})
 }
 \label{table:J}
 \end{table}
\par
Let $\Sigma_{\nu}\dual$ be the dual lattice of $\Sigma_{\nu}$,
and let $\gamma_{\nu, 1}, \dots, \gamma_{\nu, \rho(\nu)}$ be the basis of $\Sigma_{\nu}\dual$ 
dual to the basis $[C_{\nu, 1}], \dots, [C_{\nu, \rho(\nu)}]$ of $\Sigma_{\nu}$.
We also put
\[
\gamma_{\nu, 0}:=0 \;\; \in \;\; \Sigma_{\nu}\dual.
\]
For $j=0, 1, \dots, \rho(\nu)$,
we denote by $\bargamma_{\nu, j}$ the element $\gamma_{\nu, j} (\bmod\,\Sigma_{\nu})$
of  the discriminant group
$A(\Sigma_{\nu})=\Sigma_{\nu}\dual /\Sigma_{\nu}$ 
of $\Sigma_{\nu}$. 
The following is the key observation for our method:
 \begin{lemma}\label{lem:Jisom}
 The map $j\mapsto \bargamma_{\nu, j}$
 gives an isomorphism $J_{\nu}\cong \Sigma_{\nu}\dual /\Sigma_{\nu}$
 of abelian groups.
 \end{lemma}
 \begin{proof}
We compare Table~\ref{table:J} calculated in the  Kodaira--N\'eron theory 
 with the discriminant groups $\Sigma_{\nu}\dual/\Sigma_{\nu}$ of root lattices
 of type $A_{\ell}$,  $D_{\ell}$, and $E_{\ell}$.
 The order of $\Sigma_{\nu}\dual/\Sigma_{\nu}$ is classically known,
 and coincides with $|J_{\nu}|$.
We equip the vector space $\R^n$ with the standard basis $\ve_1, \dots, \ve_n$ 
and with the negative-definite intersection form  $\intf{\ve_i, \ve_j}:=-\delta_{ij}$.
 \par
 \smallskip
 {\bf The case $\tau_{\nu}=A_{\ell}$.}
 We embed $\Sigma_{\nu}$ into $\R^{\ell+1}$ by $[C_{\nu, j}]\mapsto \ve_j-\ve_{j+1}$ so that 
\[
\Sigma_{\nu}=\set{(x_1, \dots, x_{\ell+1})\in \Z^{\ell+1}\;}{\; x_1+\cdots +x_{\ell+1}=0}.
\]
Then we have 
\[
\gamma_{\nu, j}=\frac{1}{\ell+1} \left( -\sum_{k=1}^{j}  (\ell+1-j) \ve_k + \sum_{k=j+1}^{\ell+1}  j\ve_k\right)\;\; \in \;\; \Sigma_{\nu} \tensor\Q.
\]
It is easy to check that $ j \gamma_{\nu, 1} -\gamma_{\nu, j}\in \Sigma_{\nu}$.
Hence  $j\mapsto \bargamma_{\nu, j}$
 gives 
an isomorphism $\Z/(\ell+1)\Z\cong \Sigma_{\nu}\dual/\Sigma_{\nu}$.
\par
 \smallskip
 {\bf The case $\tau_{\nu}=D_{\ell}$.}
 We embed $\Sigma_{\nu}$ into $\R^{\ell}$ by 
 \[
[C_{\nu, 1}]\mapsto -\ve_1-\ve_2,\;\; 
[C_{\nu, 2}]\mapsto \ve_1-\ve_2,\;\; 
[C_{\nu, j}]\mapsto \ve_{j-1}-\ve_{j}\;\; (j=3, \dots, \ell), 
 \]
 so that we have 
 \[
\Sigma_{\nu}=\set{(x_1, \dots, x_{\ell})\in \Z^{\ell}\;}{\; x_1+\cdots +x_{\ell}\in 2\Z}.
\]
The vectors  $\gamma_{\nu, j}\in \Sigma_{\nu}\tensor\Q$ are given by 
\[
\gamma_{\nu, 1}=\frac{1}{\,2\,} \sum_{k=1}^{\ell}  \ve_k,
\quad 
\gamma_{\nu, 2}=-\frac{1}{\,2\,}\ve_1+ \frac{1}{\,2\,} \sum_{k=2}^{\ell}  \ve_k,
\quad 
\gamma_{\nu, j}=\sum_{k=j}^{\ell} \ve_{k} \;\; (j=3, \dots, \ell).
\]
It is easy to see that
$\bargamma_{\nu, 0}=0$, $\bargamma_{\nu, 1}$, $\bargamma_{\nu, 2}$, $\bargamma_{\nu, \ell}$
form the group  isomorphic,   via $\bargamma_{\nu, j}\mapsto j$, 
to the group $J_{\nu}=\{0,1,2,\ell\}$   described in  Table~\ref{table:J}.
Note that, for $j=3, \dots, \ell-1$, 
the element $\bargamma_{\nu, j}$ is 
either equal to $\bargamma_{\nu, 0}=0$ or equal to $\bargamma_{\nu, \ell}$.
 \par
 \smallskip
 {\bf The case $\tau_{\nu}=E_6$.}
 Using the basis $[C_{\nu,1}]$, \dots,  $[C_{\nu,6}]$ of $\Sigma_{\nu}$,
 we can write
\begin{eqnarray*}
 \gamma_{\nu, 2}&= &-\frac{1}{\,3\,} (3,4,5,6,4,2), \\
 \gamma_{\nu, 6}&= & - \frac{1}{\,3\,} (3,2,4,6,5,4)\;\;\equiv\;\; 2  \gamma_{\nu, 2} \;\; (\bmod\, \Sigma_{\nu}).
\end{eqnarray*}
Hence we have $\Sigma_{\nu}\dual/\Sigma_{\nu}= \{\bargamma_{\nu, 0}, \bargamma_{\nu, 2}, \bargamma_{\nu, 6}\}\cong \Z/3\Z$.
Note that we have 
$  \bargamma_{\nu, 5}= \bargamma_{\nu, 2}$,
$ \bargamma_{\nu, 3}= \bargamma_{\nu, 6}$,
$ \bargamma_{\nu, 1}= \bargamma_{\nu, 4}= \bargamma_{\nu, 0}$  in 
$\Sigma_{\nu}\dual/\Sigma_{\nu}$.
 \par
 \smallskip
 {\bf The case $\tau_{\nu}=E_7$.}
 Using the basis $[C_{\nu,1}]$, \dots,  $[C_{\nu,7}]$ of $\Sigma_{\nu}$,
 we can write
\[
 \gamma_{\nu, 7}= -\frac{1}{\,2\,} (3,2,4,6,5,4,3).
\]
Hence we have $\Sigma_{\nu}\dual/\Sigma_{\nu}= \{\bargamma_{\nu, 0}, \bargamma_{\nu, 7}\}\cong \Z/2\Z$.
Note that we have 
$  \bargamma_{\nu, 1}= \bargamma_{\nu, 5}=\bargamma_{\nu, 7}$ and 
$ \bargamma_{\nu, 2}= \bargamma_{\nu, 3}= \bargamma_{\nu, 4}=\bargamma_{\nu, 6}=\bargamma_{\nu, 0}$  in 
$\Sigma_{\nu}\dual/\Sigma_{\nu}$.
 \par
 \smallskip
 {\bf The case $\tau_{\nu}=E_8$.} Trivial.
 \end{proof}
A section $\mwsec \in \MWphi $
intersects $\phi\inv(p_{\nu})$ at a single point $\specialization_{\nu}(\mwsec)$, 
and the intersection is transverse.
Hence
the intersection point  $\specialization_{\nu}(\mwsec)$ is a smooth point of the fiber, that is,  we have $\specialization_{\nu}(\mwsec)\in  \phi\sp{*}(p_{\nu})^{\sharp}$ .
Thus we have the \emph{specialization map}
\[
\specialization_{\nu} \colon \MWphi \to  \phi\sp{*}(p_{\nu})^{\sharp}.
\]
By the definition of the group structure on $\phi\sp{*}(p_{\nu})^{\sharp}$,  
the map $\specialization_{\nu}$ is a group homomorphism.
(See~\cite[Section 5.6.1]{MWLbook}.)
The inclusion $\Sigma_{\nu}\inj S_X$ gives rise to the restriction homomorphism
$S_X\to \Sigma_{\nu}\dual$, which we write as 
\[
v\mapsto v|_{\nu}.
\]
For $\mwsec\in  \MWphi $, we have
\[
[\mwsec]|_{\nu}=\gamma_{\nu, j[\mwsec]},
\]
where $j[\mwsec]\in J_{\nu}$ is the index of the connected component of
$\phi\sp{*}(p_{\nu})^{\sharp}$ intersecting $\mwsec$,
or equivalently,
containing the point $\specialization_{\nu}(\mwsec)$.
The kernel of the composite of $S_X\to \Sigma_{\nu}\dual$ and 
$\Sigma_{\nu}\dual\to \Sigma_{\nu}\dual/\Sigma_{\nu}$ contains
the trivial sublattice $U_{\phi}\oplus \Sigma_{\phi}$. 
Hence, by Theorem~\ref{thm:ShidaTate}, 
the natural mapping 
\begin{equation}\label{eq:MWToASigma}
\MWphi\;\maprightsp{\hbox{\tiny{$[\phantom{a}]$}}} \;  S_X  \;\maprightsp{|_{\nu} } \;
 \Sigma_{\nu}\dual  \; \surj\;  \Sigma_{\nu}\dual/\Sigma_{\nu}
\end{equation}
is a group homomorphism.
By definition, 
the following diagram is commutative:
\begin{equation}\label{eq:diagmamMW}
\renewcommand{\arraystretch}{1.2}
\begin{array}{ccc}
\MWphi  &\maprightsp{\hbox{\;\tiny\eqref{eq:MWToASigma}\;\mystrutd{3.2pt}}}&\Sigma_{\nu}\dual/\Sigma_{\nu} \\
\mapdownleft{\specialization_{\nu}} & & \mapdown \rlap{$\wr$\;\; \tiny by Lemma~\ref{lem:Jisom}} \\
\phi^*(p_{\nu})^{\sharp}  &\surj&J_{\nu},
\end{array}
\end{equation}
where the lower horizontal arrow is the natural quotient homomorphism.
 \par
 Suppose that a vector $v\in S_X$ is given.
 Then the class $[s(v)]\in S_X$ of the section $s(v)\in \MWphi $ 
 corresponding to $v\bmod (U_{\phi}\oplus \Sigma_{\phi})$ by~\eqref{eq:MWisom} satisfies the following:
 \begin{enumerate}[(i)]
 \item $\intf{[s(v)], [s(v)]}=-2$ and  $\intf{[s(v)],f}=1$.
 Hence,  by the orthogonal direct-sum decomposition $S_X=U_{\phi}\oplus W_{\phi}$,
 we have $[s(v)]=t f + z +w$, 
 where $w\in W_{\phi}$ and $t=-\intf{w, w}/2$.
 \item $[s(v)]\equiv v \bmod U_{\phi}\oplus \Sigma_{\phi}$.
 In particular, for each $\nu=1, \dots, n$,
 we have 
 \[
 ([s(v)]-v)|_{\nu} \in \Sigma_{\nu}.
 \]
 \item For each $\nu=1, \dots, n$,
 there exists a unique index $j(v)\in J_{\nu}$
 such that $[s(v)]|_{\nu}=\gamma_{\nu, j(v)}$.
This $j(v)$ is the index $j$ of the connected component $C_{\nu, j}\sp{\circ}$ 
 that contains the intersection point $\specialization_{\nu}(s(v))$ of 
$s(v)$ and $\phi\inv(p_{\nu})$,
 and hence $j(v)$ is the image of $v$ by $S_X\to J_{\nu}$ in 
 the diagrams~\eqref{eq:MWToASigma}~and~\eqref{eq:diagmamMW}.
 \end{enumerate}
 Therefore the following calculations compute the class $[s(v)]$.
  \begin{enumerate}[Step 1.]
  \item Let $v\sprime\in W_{\phi}$ be the image of $v$ by the projection to $W_{\phi}$
 under  the orthogonal direct-sum decomposition $S_X=U_{\phi}\oplus W_{\phi}$.
 \item For each $\nu=1, \dots, n$,
 calculate the element $\delta_{\nu}(v\sprime):=v\sprime |_{\nu} \bmod \Sigma_{\nu}$ 
 of the discriminant group $\Sigma_{\nu}\dual/\Sigma_{\nu}$,
 and find the index $j(v)\in J_{\nu}$ such that 
 $\delta_{\nu}(v\sprime)$ is equal to $\bargamma_{\nu, j(v)}$. 
 Then the element  $ v\sprime|_{\nu}- \gamma_{\nu, j(v)}$ of $\Sigma_{\nu}\dual$ 
belongs to $\Sigma_{\nu}$.
 We calculate the integers $\alpha_{\nu, k} $ such that
 \[
 v\sprime|_{\nu}- \gamma_{\nu, j(v)}=\sum_{k=1}^{\rho(\nu)} \alpha_{\nu, k} [C_{{\nu}, k}].
 \]
 \item 
 We put
 \[
 v\spprime:=v\sprime-\sum_{\nu=1}^n \sum_{k=1}^{\rho(\nu)} \alpha_{\nu, k} [C_{{\nu}, k}].
 \]
 Then we have 
 \[
 [s(v)]=t f +z + v\spprime,
 \]
 where $t:=-\intf{v\spprime, v\spprime}/2$.
  \end{enumerate}
 %
 %
 \par
 Next, we explain how to calculate,
 for a given vector $v\in S_X$,
 the isometry 
 \[
 g(s(v))\in \OG(S_X, \PPP_X)
 \]
 induced by the translation $x\mapsto x\addE\, s(v)$ on $E_{\eta}$ by the section $s(v)\in \MWphi $,
 where $\addE$ is the addition on the elliptic curve $E_{\eta}$ over $k(\P^1)$.
 Let $m$ be the Mordell-Weil rank of $\phi$:
 \[
 m:=\dim (\MWphi \tensor\Q)=\rank S_X-2-\sum_{\nu=1}^n \rho(\nu),
 \]
 where the second equality follows from Theorem~\ref{thm:ShidaTate}.
 We choose vectors $u_1, \dots, u_m\in S_X$ such that 
 their images by 
 \[
 S_X\to (S_X/(U_{\phi}\oplus \Sigma_{\phi}))\tensor\Q
 \]
 form a basis of $ \MWphi \tensor\Q$.
 Then $S_X\tensor\Q$ is spanned by
 \begin{equation}\label{eq:Qbasis}
 \parbox{11cm}{
 $f$, \;$z=[s(0)]$, \;$[s(u_1)]$, \dots, $[s(u_m)]$, \; and the vectors 
 $[C_{\nu, 1}]$, \dots, $[C_{\nu, \rho(\nu)}]$
 in $ \Theta_{\nu}$ for $\nu=1, \dots, n$.}
 \end{equation}
 Therefore, to calculate $ g(s(v))$, it is enough to calculate the images 
 of vectors in~\eqref{eq:Qbasis} by $ g(s(v))$.
It is obvious that
\begin{eqnarray*}
f^{ g(s(v))} &= & f, \\
z^{ g(s(v))} &= & [s(v)], \\
{} [s(u_{\mu})]^{ g(s(v))} &= & [s(u_{\mu}+v)]\;\;\textrm{for}\;\; \mu=1, \dots, m.
\end{eqnarray*}
Hence it remains only to calculate the image by $ g(s(v))$  of the classes in $\Theta_{\nu}$.
Note that $ g(s(v))$ induces a permutation on 
the set $\tilTheta_{\nu}=\{[C_{\nu, 0}]\} \cup \Theta_{\nu} $
that preserves the subset $J_{\nu}$ of 
classes of reduced irreducible components. 
By the method described in Step 2 above,
we calculate the index $j(v)\in J_{\nu}$,
which is the image of $s(v)\in \MWphi $ 
by the composite of $\specialization_{\nu}\colon \MWphi \to \phi^*(p_{\nu})\sp{\sharp}$ and 
$\phi^*(p_{\nu})\sp{\sharp} \to J_{\nu}$.
The translation of $\phi^*(p_{\nu})\sp{\sharp}$ by $\specialization_{\nu}(s(v))$
induces the translation of $J_{\nu}$ by $j(v)$.
Checking each Dynkin diagram of type $A_{\ell}$, $D_{\ell}$, $E_{\ell}$,
we see that this permutation of $J_{\nu}$
extends \emph{uniquely} to a permutation of $\tilTheta_{\nu}$
that preserves the dual graph.
See Table~\ref{table:permj},
in which we abbreviate  $\tilTheta_{\nu}=\{[C_{\nu, 0}]\,\dots, [C_{\nu, \rho(\nu)}]\}$ 
as $\{0, 1, \dots, \rho(\nu)\}$.
Hence the image of each element of $\tilTheta_{\nu}$ 
by $ g(s(v))$ is computed.
Using~\eqref{eq:Cnu0},
we can calculate the action of  $ g(s(v))$  on the classes of $\Theta_{\nu}$.
\begin{table}
\[
\renewcommand{\arraystretch}{1.2}
\begin{array}{llll}
\tau_{\nu} & J_{\nu} & j (v) & \textrm{Permutation of $\tilTheta_{\nu}$} \\
\hline
A_{\ell} &\Z/(\ell+1)\Z& a & i \mapsto (i+a) \bmod (\ell+1) \mystruthd{15pt}{10pt}\\
\hline
D_{\ell} & (\Z/2\Z)^2 &0 & \id \\
(\ell: \textrm{even})  && 1 & 0\leftrightarrow 1,
\quad
2\leftrightarrow \ell,
\quad k \leftrightarrow \ell+2-k \;\;(2<k<\ell)\\
&& 2 &0\leftrightarrow 2,
\quad
1\leftrightarrow \ell,
\quad k \leftrightarrow \ell+2-k \;\;(2<k<\ell)
 \\
&& \ell &  0\leftrightarrow \ell,
\quad
1\leftrightarrow 2,
\quad k \leftrightarrow k  \;\;(2<k<\ell) \\
\hline
D_{\ell} & \Z/4\Z& 0 & \id \\
(\ell: \textrm{odd}) &&1& 0 \mapsto  1 \mapsto \ell \mapsto 2 \mapsto 0,
\quad k \leftrightarrow \ell+2-k  \;\;(2<k<\ell) \\
&& 2 & 0 \mapsto 2 \mapsto \ell \mapsto 1 \mapsto 0,
\quad k \leftrightarrow \ell+2-k   \;\;(2<k<\ell) \\
&& \ell &
0\leftrightarrow \ell,
\quad
1\leftrightarrow 2,
\quad 
k \leftrightarrow k  \;\;(2<k<\ell) \\
\hline
E_6 & \Z/3\Z & 0 & \id \\
&& 2 & 0\mapsto 2\mapsto 6 \mapsto 0, \quad 1 \mapsto 3 \mapsto 5 \mapsto 1, \quad 4\mapsto 4 \\
&& 6 & 0\mapsto 6\mapsto 2 \mapsto 0, \quad 1 \mapsto 5 \mapsto 3 \mapsto 1, \quad 4\mapsto 4 \\
\hline
E_7 & \Z/2\Z & 0 & \id \\
&& 7 & 0  \leftrightarrow 7, \quad 1  \leftrightarrow 1,\quad  4  \leftrightarrow 4,\quad 2  \leftrightarrow 6,\quad 3  \leftrightarrow 5\quad  \\
\hline
E_8 &0 & 0 & \id \\
\end{array}
\]
\vskip .5cm
\caption{Permutations of $\tilTheta_{\nu}$}\label{table:permj}
\end{table}
\section{Borcherds' method}\label{sec:Borcherds}
\subsection{An algorithm on a graph}\label{subsec:VE}
We recall an algorithm introduced in~\cite{BrandhorstShimada2021}.
Let $(V, E)$ be a simple non-oriented connected graph,
where $V$ is the set of vertices
and $E$ is the set of edges,
which is a set of non-ordered pairs of distinct elements of $V$:
\[
E\subset \dbinom{V}{2}.
\]
We say that $v, v\sprime\in V$ are \emph{adjacent} if $\{v, v\sprime\}\in E$.
The set  $V$ may be infinite.
The assumption that $(V, E)$ be connected is important.
Suppose that a group $G$ acts on $(V, E)$ from the right.
For vertices $v, v\sprime\in V$,  we put
\begin{equation*}\label{eq:TG}
\TG (v, v\sprime):=\set{g\in G}{v^g=v\sprime}, 
\end{equation*}
and define the \emph{$G$-equivalence relation} $\sim$ on $V$ by
\[
v\sim v\sprime\;\;\Longleftrightarrow\;\; \TG (v, v\sprime)\ne\emptyset.
 \]
 Thus we have two relations on $V$, 
 the adjacency relation and the $G$-equivalence relation.
Suppose that $V_0$ is
a non-empty subset of $V$ with the following properties.
\begin{enumerate}[(a)]
\item If $v, v\sprime \in V_0$ are distinct,
then $v$ and $ v\sprime$ are not $G$-equivalent.
\item
If a vertex $v\in V$ is adjacent to a vertex in $V_0$,
then $v$ is $G$-equivalent to a vertex in $V_0$.
\end{enumerate}
We put
\[
\widetilde{V}_0:=\set{v\in V}{ \textrm{$v$ is adjacent to a vertex in $V_0$}}.
\]
Then, for each  $v\in \widetilde{V}_0$,
there exists a  vertex $u_0(v)\in V_0$  such that  
$\TG (v, u_0(v))\ne \emptyset$.
Note that  
 $u_0(v)\in V_0$  is unique by assumption (a).
We choose an element $h(v)$ from $ \TG (v, u_0(v))$
for each  $v\in \widetilde{V}_0$, and 
put
\begin{equation}\label{eq:HHH}
\HHH:=\set{h(v)}{v\in \widetilde{V}_0}.
\end{equation}
\begin{proposition}[Proposition 4.1 of \cite{BrandhorstShimada2021}]\label{prop:VE}
The subset $V_0\subset V$ is a complete set of representatives of 
the orbit decomposition of $V$ by $G$, 
and the group $G$ is generated by the union of $\HHH$
 and the stabilizer subgroup 
$\Stab_G(v_0)=\TG (v_0, v_0)$
of a vertex $v_0\in V_0$.
\qed
\end{proposition}
In~\cite[Section 4.1]{BrandhorstShimada2021},
we presented an algorithm to obtain $V_0$ and $\HHH$
under the assumption that $(V, E)$ and $G$ have 
certain \emph{local effectiveness properties}.
\subsection{Period condition}
In this subsection,
we assume that the base field $k$ is the complex number field $\C$,
and introduce \emph{period condition} on elements of $\OG(S_X)$.
The period condition is, 
however, also defined when $X$ is a supersingular $K3$ surface in positive characteristic.
See, for example,~\cite{KondoShimada2014b}.
\par
Let $L$ be an even lattice, and $A(L)=L\dual/L$ the discriminant group of $L$.
We define a quadratic form 
\[
q(L)\colon A(L)\to \Q/2\Z
\]
by $q(x \bmod L):=\intf{x,x} \bmod 2\Z$.
This finite quadratic form is called the \emph{discriminant form} of $L$,
which was introduced by Nikulin~\cite{Nikulin1979}.
Let $M$ be a primitive sublattice of $L$, and $N$ the orthogonal complement of $M$ in $L$.
Then we have natural embeddings 
\[
M\oplus N\;\;\subset\;\; L\;\;\subset \;\; L\dual\;\;\subset\;\; M\dual\oplus N\dual.
\]
Suppose that $L$ is unimodular, that is, $L\dual=L$.
Then the submodule 
\[
L/(M\oplus N)\;\;\subset \;\; A(M)\times A(N)
\]
is a graph of an isomorphism $A(M)\cong  A(N)$,
which induces an isomorphism
\begin{equation*}\label{eq:qsisom}
\iota_L\colon q(M)\cong -q(N).
\end{equation*}
Nikulin~\cite{Nikulin1979} proved the following.
\begin{proposition}\label{prop:NikulinExtend}
Suppose that $L$ is unimodular.
Let $G_N$ be a subgroup of $\OG(N)$,
and let $q(G_N)\subset \Aut(q(N))$ be the image of $G_N$
by the natural homomorphism $\OG(N)\to \Aut(q(N))$.
Then an isometry $g_M$ of $M$ extends to an isometry $g_L$ of $L$
such that its restriction $g_L|N$ to $N$ is an element of $G_N$ if and only if 
the action of $g_M$ on $q(M)$ belongs to $q(G_N)$ via the isomorphism
$\Aut(q(M))\cong \Aut(q(N))$ induced by $\iota_L\colon q(M)\cong -q(N)$.
\qed
\end{proposition}
\par
We apply this result to the primitive embedding of $S_X$ into  
the even unimodular lattice $H^2(X, \Z)$ of rank $22$
defined by the cup product.
Let $T_X$ denote the orthogonal complement of $S_X$ in $H^2(X, \Z)$,
which we call the \emph{transcendental lattice} of $X$.
Then $H^2(X, \Z)$  induces an isomorphism 
\[
\iota_{H}\colon q(S_X)\cong -q(T_X).
\]
Note that $T_X$ is the minimal primitive submodule of $H^2(X, \Z)$
such that $T_X\tensor \C$ contains the period $H^{2,0}(X)=\C \omega_X \subset H^2(X, \C)$ of $X$,
where $\omega_X$ is a nonzero holomorphic $2$-form on $X$.
\begin{definition}\label{def:period}
We put 
\[
\OG(T_X, \omega_X):=\set{g_T\in \OG(T_X)}{\textrm{$g_T\tensor\C$ preserves $H^{2,0}(X)$}}.
\]
Then we say that $g_S \in \OG(S_X)$ satisfies  the \emph{period condition} if 
the action  of $g_S$ on $q(S_X)$ is equal to the action  on $q(T_X)$ of some of $g_T\in \OG(T_X, \omega_X)$
via the isomorphism $\iota_{H} \colon q(S_X)\cong -q(T_X)$ induced by $H^2(X, \Z)$.
\end{definition}
By  Proposition~\ref{prop:NikulinExtend},
we see that an  isometry $g_S \in \OG(S_X)$ extends to an isometry of $H^2(X,\Z)$ preserving the period $H^{2,0}(X)$
if and only if $g_S$ satisfies the period condition.
By Torelli theorem~\cite{Torelli1971}
(see also~\cite[Chapter VIII]{BHPV2004}), we obtain the following:
\begin{theorem}\label{thm:Torelli}
We put
\[
G:=\Image(\Aut(X)\to \OG(S_X, \PPP_X)).
\]
Then $g\in \OG(S_X, \PPP_X)$ belongs to $G$ if and only if 
$g $ preserves $N_X$ and satisfies the period condition.
\qed
\end{theorem}
\begin{example}\label{example:simple_period_condition}
Suppose that $\rank T_X\ge 3$ and that
$\omega_X$ is very general in the period domain $\QQQ$ in  $\P_{*}(T_X\tensor \C)$.
(See~\cite[Chapter VIII]{BHPV2004} for the definition of the period domain.)
Then we have 
\begin{equation}\label{eq:periodTpm1}
\OG(T_X, \omega_X)=\{\pm 1\},
\end{equation}
and hence $g_S\in \OG(S_X)$ satisfies  the period condition
if and only if the action of $g_S$ on the discriminant group $A(S_X)$ is $1$ or $-1$.
\par
We give a proof of~\eqref{eq:periodTpm1}.
The period domain $\QQQ$ is an open subset (in the classical topology)
of a smooth quadratic hypersurface in $\P_{*}(T_X\tensor \C)$,
and hence we have $\dim \QQQ=\rank T_X-2 >0$.
For $\gamma\in \OG(T_X)$,
let $V_{\gamma, \lambda}\subset T_X\tensor \C$ denote the eigenspace of $\gamma$ with eigenvalue $\lambda\in \C$.
If $\gamma\notin \{\pm 1\}$, then $\dim V_{\gamma, \lambda}<\rank T_X$
and hence $\P_*(V_{\gamma, \lambda})\cap \QQQ$ is a proper analytic subspace of $\QQQ$ for any $\lambda$.
Since a countable union of proper analytic subspaces of a positive-dimensional 
connected complex manifold cannot cover the total space,
we have~\eqref{eq:periodTpm1} for  $\omega_X$ very general in $\QQQ$.
\par
Suppose moreover that $-1\in \OG(T_X, \omega_X)$ acts 
on  $A(T_X)$ non-trivially
(that is, the abelian group $A(T_X)\cong A(S_X)$ is not $2$-elementary).
By Proposition~\ref{prop:NikulinExtend}, 
there exists no isometry $g_H$ of the overlattice $H^2(X, \Z)$ of $S_X\oplus T_X$
such that $g_H|S_X=1$ and $g_H|T_X=-1$.
Since $\Aut(X)$ acts on $H^2(X, \Z)$ faithfully,
the natural homomorphism
$\Aut(X)\to \OG(S_X, \PPP_X)$ is injective.
\end{example}
\begin{remark}
For supersingular $K3$ surfaces,
we have to prove~\eqref{eq:periodTpm1} in a different method,
because the period domain is a subvariety of codimension $>1$ in a Grassmannian variety.
See~\cite{KondoShimada2014b}.
\end{remark}
\subsection{Tessellation by $\LS$-chambers}
Let $L_{26}$ denote an even unimodular hyperbolic lattice 
of rank $26$,
which is unique up to isomorphism.
We choose a positive cone $\PPP_{26}$ of $L_{26}$.
A standard fundamental  domain  of 
$W(L_{26})$  
was determined by Conway~\cite{Conway1983}
by means of Vinberg's algorithm~\cite{VinbergBombay}.
\begin{definition}
A vector $\weyl\in L_{26}$ is called a \emph{Weyl vector} if 
$\weyl$ is a non-zero primitive vector of $L_{26}$ contained in
$\bdr\closure{\PPP}_{26}$  
(in particular, we have $\intf{\weyl, \weyl}=0$ and hence 
$\Z \weyl \subset (\Z \weyl )\sperp$) such that 
$(\Z \weyl )\sperp/\Z \weyl  $ is isomorphic to 
the negative-definite Leech lattice. 
\end{definition}
%
%
\begin{definition}
Let $\weyl$ be a Weyl vector.
A $(-2)$-vector $r\in L_{26}$ is said to be a \emph{Leech root} 
with respect to  $\weyl$
if $\intf{\weyl, r}=1$.
We then put 
\[
\ConC (\weyl):=\set{x\in \PPP_{26}}{\intf{x, r}\ge 0\;\;\textrm{for all
Leech roots $r$ with respect to $\weyl$}}.
\]
\end{definition}
\begin{theorem}[Conway~\cite{Conway1983}]
\begin{enumerate}[{\rm (1)}]
\item The mapping $\weyl\mapsto \ConC (\weyl)$ gives
a bijection from the set of Weyl vectors 
to the set of standard fundamental domains of $W(L_{26})$.
\item
Let $\weyl$ be a Weyl vector.
Then the mapping $r\mapsto \ConC (\weyl)\cap (r)\sperp$ gives a bijection from 
the set  
of Leech roots with respect to $\weyl$  to the set of 
 walls of the  chamber $\ConC (\weyl)$. 
 \qed
\end{enumerate}
\end{theorem}
\begin{definition}
We call 
a standard fundamental domain
of $W(L_{26})$   
a \emph{Conway chamber}.
Hence $\PPP_{26}$ is tessellated
by the Conway chambers.
\end{definition}
Suppose that we have a primitive embedding
\[
\iota\colon S_X\inj L_{26}.
\]
Replacing $\iota$ by $-\iota$ if necessary,
we assume that $\iota$ maps $\PPP_X$ into $\PPP_{26}$, and 
regard   $\PPP_X$
as a subspace of $\PPP_{26}$:
\[
\PPP_X=\iota\inv(\PPP_{26})= (S_X\tensor \R)\cap \PPP_{26}.
\]
\begin{definition}
An  \emph{$L_{26}/S_X$-chamber} 
is a chamber $D$ of $\PPP_X$ that is obtained as the intersection $\PPP_X\cap \ConC (\weyl)$ of
$\PPP_X$ with a Conway chamber $\ConC (\weyl)$.
\end{definition}
The tessellation of $\PPP_{26}$ by the Conway chambers
induces a tessellation of $\PPP_X$ by the $L_{26}/S_X$-chambers.
By definition, the nef-and-big cone $N_X$,
which is a standard fundamental domain of $W(S_X)$,
 is tessellated by $L_{26}/S_X$-chambers.
 In other words, 
 the tessellation of $\PPP_X$ by the $L_{26}/S_X$-chambers is a refinement of 
 the tessellation 
 by the standard fundamental domains of $W(S_X)$.
 \begin{definition}\label{def:thegraphVE}
We define a graph $(V, E)$  by the following.
\begin{itemize}
\item The set $V$ of vertices is the set of $L_{26}/S_X$-chambers contained in $N_X$.
\item The set $E$ of edges is the set of pairs of adjacent $L_{26}/S_X$-chambers.
\end{itemize}
 \end{definition}
Let $G$ be the image of 
the natural homomorphism $\Aut(X)\to \OG(S_X, \PPP_X)$.
Suppose that  
 \begin{equation}\label{cond:period}
 \parbox{10cm}{
the period condition for $g\in \OG(S_X)$ 
is that the action of $g$ on the discriminant group $A(S_X)$ be $1$ or $-1$.}
 \end{equation}
See Example~\ref{example:simple_period_condition}
for a case where this assumption is satisfied.
Then,
by Proposition~\ref{prop:NikulinExtend},
every element $g\in G$ extends to an isometry of $L_{26}$.
In particular,
the action of $G$ preserves the tessellation of $\PPP_X$ by the $\LS$-chambers.
Since the action of $G$ preserves $N_X$, we obtain the following:
\begin{proposition}\label{eq:GoactsonVE}
If~\eqref{cond:period} holds, then  
 $G$ acts on the graph $(V, E)$.
 \qed
\end{proposition}
\begin{definition}\label{def:primitive}
Let $D=\PPP_X \cap \ConC (\weyl)$
be an $L_{26}/S_X$-chamber.
For each wall $w$ of $D$,
there exists a unique defining  vector $v$ of $w$  in the dual lattice $S_X\dual$
that is primitive in $S_X\dual$.
We call this vector $v\in S_X\dual$ the \emph{primitive defining vector} of the wall $w$.
\end{definition}
Note that a Conway chamber has infinitely many walls.
For the graph $(V, E)$ to have local effectiveness properties in~\cite{BrandhorstShimada2021},
it needs that each $\LS$-chamber has only a finite number of walls.
 We consider the following assumption:
 \begin{equation}\label{cond:atleastoneroot}
 \parbox{10cm}{
The orthogonal complement of $S_X$ in $L_{26}$ cannot be embedded in the negative-definite Leech lattice.}
 \end{equation}
This holds, for example, if the orthogonal complement
contains at least one  $(-2)$-vector.
\begin{proposition}[\cite{Shimada2015}]\label{prop:finitewalls}
Suppose that~\eqref{cond:atleastoneroot} holds.
Then each $L_{26}/S_X$-chamber has only a finite number of walls.
If $D=\PPP_X \cap \ConC (\weyl)$ is an  $L_{26}/S_X$-chamber
obtained by the Conway chamber $\ConC (\weyl)$
associated with a Weyl vector $\weyl$,
then we can calculate the primitive defining vectors of walls of $D$ from $\weyl$.
Moreover, for each wall $w$ of $D$,
we can calculate a Weyl vector $\weyl\sprime$
such that $D\sprime=\PPP_X \cap \ConC (\weyl\sprime)$ is the $L_{26}/S_X$-chamber
adjacent to $D$ across the wall $w$.
\qed
\end{proposition}
Thus, under assumptions~\eqref{cond:period}~and~\eqref{cond:atleastoneroot},
the local effectiveness properties in~\cite{BrandhorstShimada2021}
hold for $(V, E)$ and $G$, and 
we can apply the algorithm in~\cite[Section 4.1]{BrandhorstShimada2021}  
to $(V, E)$ and $G$.
%
%
\begin{remark}\label{rem:V0}
The amount of the computation of this method is estimated by
$|V_0|=|V/G|$,
that is, the number of the orbits of the action of $\Aut(X)$ on the set of $L_{26}/S_X$-chambers
contained in $N_X$.
\par
In practice, it seems that Borcherds' method 
carried out without using computer (for example,~\cite{Kondo1998})
can only deal with the case where $|V_0|=1$.
Some  cases with $|V_0|>1$
were treated in~\cite{Shimada2015},
where $V_0$ is of size about $10^3\sim 10^4$.
However, the geometric description of the generators of $\Aut(X)$ was not given
for these cases.
We also have observed 
some cases where $|V_0|$ is too large for  Borcherds' method to terminate in a reasonable time
(for example,~\cite{KKS2014}).
\par
In the case of the present article (see Section~\ref{sec:AutXfg}),
we have $|V_0|=7$.
Since this is not so large,
we have managed to obtain geometric generators.
\end{remark}
\begin{remark}
It has been \emph{empirically} observed that $|V_0|$ is small when the orthogonal complement of $\iota\colon S_X\inj L_{26}$ contains a root lattice
as a  sublattice of finite index.
\end{remark}
\section{Computation of $\Aut(\Xfg)$}\label{sec:AutXfg}
In this section, we prove Theorems~\ref{thm:genssimple}~and~\ref{thm:rats}.
For simplicity, 
we write $X$ for 
the $K3$ surface $\Xfg$.
Recall that 
the polynomials $f$ and $g$ 
in the defining equation~\eqref{eq:wfg} of $\barX_{f, g}$ are assumed to be very general.  
We use this assumption throughout this section.
\subsection{The lattice $S_X$}
First, we describe the lattice $S_X$ and the nef-and-big cone $N_X$.
Let $H\subset X$ denote the pull-back of a line of $\P^2$,
and we put
\[
\theh:=[H]\in S_X.
\]
The singular locus  
of the branch curve $B(\theh)=\{f^2+g^3=0\}\subset \P^2$ of the finite double covering $\barX_{f,g}\to \P^2$ 
consists of six ordinary cusps $\bar{p}_1, \dots, \bar{p}_6$, which are located at 
the locus defined by $f=g=0$.
Hence the singularities  
of $\barX_{f,g}$ consist of six rational double points $p_1, \dots, p_6$ of type $A_2$,
where $p_i$ is located over $\bar{p}_i$.
Let $E_i\spar{+}$ and $E_i\spar{-}$ denote the exceptional curves that are 
contracted to the point $p_i\in \Sing(\barX_{f,g})$
by the desingularization $X\to \barX_{f,g}$.
We put
\[
\ve_i\spar{+}:=[E_i\spar{+}]\;\in\; S_X, \quad \ve_i\spar{-}:=[E_i\spar{-}] \;\in\; S_X.
\]
Let $\barGamma\subset \P^2$ be the conic defined by $g=0$.
Then $\barGamma$ passes through the six cusps $\bar{p}_1, \dots, \bar{p}_6$ of $B(\theh)$.
Hence the strict transform of $\barGamma$ in $X$ is a disjoint union of two smooth rational curves 
$\Gamma\spar{+}$ and $\Gamma\spar{-}$.
We put
\[
\vgamma\spar{+}:=[\Gamma\spar{+}]\;\in\; S_X, \quad \vgamma\spar{-}:=[\Gamma\spar{-}] \;\in\; S_X.
\]
For each $i\in \{1,\dots, 6\}$,
the curve $\Gamma\spar{+}$ intersects one  of $E_i\spar{+}$ or $E_i\spar{-}$
and is disjoint from the other.
Interchanging $E_i\spar{+}$ and $E_i\spar{-}$ if necessary,
we can assume that 
\[
\intf{\vgamma\spar{+}, \ve_i\spar{+}}=1, \quad 
\intf{\vgamma\spar{+}, \ve_i\spar{-}}=0
\]
holds for $i=1, \dots, 6$.
Then we have the following. (See also~\cite{Shimada2010}.)
\begin{proposition}[Degtyarev~\cite{Degtyarev2008}]\label{prop:SX}
The $\Q$-vector space $S_X\tensor\Q$ is of dimension $13$, and is generated by the classes
\begin{equation}\label{eq:bone}
\theh,\;\;  \ve_1\spar{+},  \ve_1\spar{-}, \quad\dots\quad , \ve_6\spar{+},  \ve_6\spar{-}.
\end{equation}
The sublattice $S_{X, 0}$ of $S_X$ generated by the classes in~\eqref{eq:bone} 
is of index $3$ in $S_X$.
The lattice $S_X$ is generated by $S_{X, 0}$ and the class $\vgamma\spar{+}$.
\qed
\end{proposition}
By Proposition~\ref{prop:SX}, 
a vector $v$ of  $S_X\tensor \Q$ is uniquely determined  by the list of intersection numbers 
\[
\intf{v, \theh}, \; 
\intf{v,  \ve_1\spar{+}},
\;  \intf{v,  \ve_1\spar{-}}, \dots\,  \intf{v, \ve_6\spar{+}}, \; \intf{v, \ve_6\spar{-}}.
\]
Moreover,
an isometry $g$ of $S_X$ is specified by the images of the classes in~\eqref{eq:bone} by $g$.
For example, 
the involution $ \invol(\theh)$ associated with 
the double covering $\pi(\theh)\colon X\to \P^2$ defined by $|\theh|$ is 
 given by
\[
\theh^{\invol (\theh)}=\theh, \quad  (\ve_i\spar{+})^{\invol (\theh)}= \ve_i\spar{-},  \;  (\ve_i\spar{-})^{\invol(\theh)}= \ve_i\spar{+} \quad (i=1, \dots, 6).
\]
The vector $\ampleX \in S_X\tensor \Q$ defined by   
\begin{equation}\label{eq:theample}
\intf{\ampleX , \theh}=8, \;\;\;  \intf{\ampleX ,  \ve_i\spar{+}}=1,\; \intf{\ampleX ,  \ve_i\spar{-}}=1
\quad (i=1, \dots, 6)
\end{equation}
is a vector of $\PPP_X\cap S_X$, and satisfies 
\[
\intf{\ampleX , \ampleX }=20, \quad \Roots( \,\ampleX \sperp\cap S_X\,)=\emptyset, \quad \Sep(\theh, \ampleX )=\emptyset.
\]
Hence $\ampleX $ is  ample  (see~Section~\ref{subsec:findanample}).
By this ample class $\ampleX$, we can specify the nef-and-big cone $N_X$ in $\PPP_X$.
\par
Next, we investigate the period condition of $X$.
We consider the moduli space $\MMM$ of lattice-polarized $K3$ surfaces
$(X\sprime, \eta\sprime)$,
where $X\sprime$ is a $K3$ surface and $\eta\sprime$ is an isometry $H^2(X, \Z)\cong H^2(X\sprime, \Z)$
that induces an embedding $S_X \inj S_{X\sprime}$.
Then $\MMM$ is covered by the period domain $\QQQ\subset \P_*(T_X\tensor\C)$.
If $(X\sprime, \eta\sprime)$ is very general in $\MMM$, 
then we have $S_X=S_{X\sprime}$.
Looking at the lattice $S_X=S_{X\sprime}$,
we obtain the following:
\begin{proposition}[Degtyarev~\cite{Degtyarev2008}]\label{prop:fsprimegsprime}
If $(X\sprime, \eta\sprime)$ is very general in $\MMM$, then there exist homogeneous polynomials $f\sprime$ and $g\sprime$  
of degree $3$ and $2$, respectively, such that
$X\sprime$ is birational to the double plane defined by $w^2=f^{\prime 2} +g^{\prime 3}$.
\qed
\end{proposition}
\begin{remark}
The following naive dimension count may help in understanding Proposition~\ref{prop:fsprimegsprime}:
the  dimension of the parameter space of pairs $(f\sprime, g\sprime)$ 
of homogeneous polynomials of degree $3$ and $2$ 
modulo linear transformation
is equal to
\[
\dim H^0(\P^2, \OOO(3))+ \dim H^0(\P^2, \OOO(2))-\dim \GL(3, \C)
=7=\rank T_X-2=\dim\QQQ.
\]
See also~\cite{Shimada2010} for the proof of Proposition~\ref{prop:fsprimegsprime}.
\end{remark}
Since $f$ and $g$ are very general, 
we see that $X$ is very general in $\MMM$, and hence  
we can assume that $\omega_X$ is very general in the period domain $\QQQ$.
Therefore, by~\eqref{eq:periodTpm1} in Example~\ref{example:simple_period_condition},  we have
\begin{equation}\label{eq:OTomega}
\OG(T_X, \omega_X)=\{\pm 1\}.
\end{equation}
The discriminant group $A(S_X)$ of $S_X$  is isomorphic to $\Z/2\Z\times (\Z/3\Z)^4$.
Hence, by Example~\ref{example:simple_period_condition}, 
we obtain the following: 
\begin{proposition}\label{prop:inj}
The natural representation of 
$\Aut(X)$ on $S_X$ is faithful.
\qed
\end{proposition}
We will consider $\Aut(X)$ as a subgroup of $\OG(S_X, \PPP_X)$  from now on.
By~Theorem~\ref{thm:Torelli}  and~\eqref{eq:Sepaag}, we have the following:
\begin{proposition}
An element  $g\in \OG(S_X, \PPP_X)$ belongs to 
$\Aut(X)$
if and only if $g$  acts on $A(S_X)$ as $1$ or $-1$,  and\, $\Sep(\ampleX , \ampleX ^{g})=\emptyset$ holds.
\qed
\end{proposition}
We introduce an auxiliary group $M$, which
makes the descriptions of $N_X$ and $\Aut(X)$ much easier.
Let $M$ be the subgroup of $\OG(S_X, \PPP_X)$
consisting of  elements $g$ satisfying $\theh^g=\theh$ and  
\[
\{\;\ve_1\spar{+}, \ve_1\spar{-},\dots, \,\ve_6\spar{+}, \ve_6\spar{-}\;\}^g
=\{\;\ve_1\spar{+}, \ve_1\spar{-},\dots, \,\ve_6\spar{+}, \ve_6\spar{-}\;\}.
\]
Then $M$ is isomorphic to $\Z/2\Z\times S_6$,
generated by the involution $\invol(\theh)$ and permutations $\sigma\in S_6$ given by
\[
\theh^{\sigma}=\theh, \quad \ve_i\sp{(+) \sigma}= \ve_{i^{\sigma}}\sp{(+)},  \quad \ve_i\sp{(-)\sigma}= \ve_{i^{\sigma}}\sp{(-)}.
\]
For  each $g\in M$, we have   $\ampleX = \ampleX ^{g}$, 
and  hence  $M\subset \OG(S_X, N_X)$.
The discriminant form $q(S_X)$ of $S_X$   is isomorphic to
\[
\left(  \left[\frac{\,1\,}{2}\right], \Z/2\Z   \right)
\;\oplus\; 
\left(  \left[\frac{\,4\,}{3}\right], \Z/3\Z \right)^{\oplus 3}
\;\oplus\;
 \left( \left[\frac{\,2\,}{3}\right], \Z/3\Z \right).
\] 
Here $([\alpha], \Z/m\Z)$ denotes a cyclic group $A=\gen{\gamma}$ of order $m$
generated by $\gamma$ equipped with the quadratic form $q\colon A\to \Q/2\Z$ 
such that $q(\gamma)=\alpha$.
The natural homomorphism $\OG(S_X)\to \Aut(q(S_X))$ maps $M$ to $ \Aut(q(S_X))$ isomorphically.
Note that $\iota(\theh)$ acts on $A(S_X)$ as  $-1$.
Hence we have 
\[
M\cap \Aut(X)=\{1, \invol(\theh)\}.
\]
\begin{remark}\label{rem:rats}
By means of the methods in Section~\ref{subsec:Rats}, 
we can make the list of classes of smooth rational curves $C$ on $X$ with $\intf{[C],  \theh}=m$
for each non-negative integer $m$.
The size $\nu(m)$ of this list is as follows:
When $m$ is odd, we have $\nu(m)=0$, whereas for $m$ even, we have 
\[
\renewcommand{\arraystretch}{1.4}
\begin{array}{c| cccccccc}
m & 0& 2 & 4 & 6 & 8 & 10 & 12 & 14 \\
\hline
\nu(m) & 12 & 17 &0 & 492 & 720 &  492 & 8292  &8730 
\end{array}.
\]
For $i, j$ with $1\le i\le 6$, $1\le j\le 6$, and $i\ne j$,
let $\ell_{ij}\subset \P^2$ denote the line passing through 
the singular points $\bar{p}_i$ and $\bar{p}_j$ of the branch curve $B(\theh)$,
and let $\tilell_{ij}\subset X$ be the strict transform of $\ell_{ij}$.
The $\nu(2)=17$ smooth rational curves on $X$ of degree $2$ with respect to $\theh$ are 
the lifts $\Gamma\spar{\pm}$ of the conic
 $\barGamma\subset \P^2$ 
and the curves $\tilell_{ij}$.
\end{remark}
\subsection{Automorphisms of $X$}\label{subsec:auttype}
By the method in Section~\ref{subsec:findingaut},
we find many automorphisms of $X$ from nef vectors of norm $2$.
Among them, we have the following automorphisms:
\begin{enumerate}[{\;\;\rm type (a):}]
\item the involution $\invol(\theh)$, 
\item $90$ involutions $\invol(h_{IJ})$ associated with polarizations $h_{IJ}$ of degree $2$ such that $\intf{h_{IJ}, \theh}=6$
and that $\Sing(B(h_{IJ}))$ is of type $A_3+A_5$,
\item $12$ involutions $\invol(\haa\sp{\pm})$ associated with polarizations $\haa\sp{\pm}$ 
 of degree $2$  such that $\intf{\haa\sp{\pm}, \theh}=4$
 and that $\Sing(B(\haa\sp{\pm}))$ is of type $A_2+5 A_1$,
\item $360$ involutions $\invol(h_{\pm J})$ associated with 
polarizations  $h_{\pm J}$  of degree $2$  such that $\intf{h_{\pm J}, \theh}=14$, 
 and that $\Sing(B(h_{\pm J}))$ is of type $D_4+A_5$, and 
\item $360$ translations associated with  sections $\ve_j\spar{\pm}$ of infinite order of $120$ Jacobian fibrations $\phi\colon X\to \P^1$
defined by $(f_{\phi}, z_{\phi})=(f_{\pm I}, \ve_i\spar{\pm})$ with $\intf{f_{\pm I}, \theh}=4$
such that $ \MW_{\phi}$ is torsion-free of rank $4$
and that the reducible fibers of $\phi$ are of type  $D_4+A_3$.
\end{enumerate}
See subsections below for more precise descriptions of these automorphisms.
We will show,  by Borcherds' method, that these automorphisms generate $\Aut(X)$.
\subsection{Primitive embedding $S_X\inj L_{26}$}
To apply Borcherds' method, 
we embed $S_X$ into $L_{26}$ primitively.
Let $R_0$ be a negative-definite root lattice of type $A_1+6A_2$
with a  basis 
\begin{equation}\label{eq:basisR0}
\alpha, \;\beta_1\spar{+}, \beta_1\spar{-}, \; \dots,  \; \beta_6\spar{+}, \beta_6\spar{-}
\end{equation}
consisting of roots that form the dual graph as in Figure~\ref{fig:basisR0}.
Let
 \[
\alpha\dual, \;\beta_1\sp{(+)\vee}, \beta_1\sp{(-)\vee}, \; \dots,  \; \beta_6\sp{(+)\vee}, \beta_6\sp{(-)\vee}
\]
be the basis of the dual lattice $R_0\dual$ that is dual to the basis~\eqref{eq:basisR0}.
Then
\[
R:=R_0 + \Z \left(  \beta_1 \sp{(+)\vee}+ \cdots + \beta_6 \sp{(+)\vee} \right)\;\;\subset\;\; R_0\dual
\] 
is an even lattice whose discriminant form is isomorphic to $-q(S_X)$.
Recall that 
 the natural homomorphism  $ \OG(S_X)\to \Aut(q(S_X))$ maps $M$ to $\Aut(q(S_X))$ isomorphically, and hence is surjective.
Therefore,
by Nikulin~\cite{Nikulin1979},
 there exists a \emph{unique} (up to the action of $\OG(S_X)$) even unimodular overlattice of $S_X\oplus R$
in which $S_X$ and $ R$ are both primitive.
Taking this unimodular overlattice as $L_{26}$, 
we find a primitive embedding
\[
\iota\colon S_X\inj L_{26}.
\]
We consider the tessellation  of $N_X\subset \PPP_X$  by the $\LS$-chambers
associated with this primitive embedding.
Let $(V, E)$ be the graph of $\LS$-chambers contained in $N_X$
(see Definition~\ref{def:thegraphVE}).
By~\eqref{eq:OTomega} and Propositions~\ref{eq:GoactsonVE},~\ref{prop:finitewalls}, 
we see  that 
the group $G=\Aut(X)\subset \OG(S_X, \PPP_X)$ acts on the graph $(V, E)$,
and we can apply the algorithm in~\cite[Section 4.1]{BrandhorstShimada2021}.
\begin{remark}
Primitive embeddings of $S_X$ into $L_{26}$ are not unique.
In fact, the genus of negative-definite even lattices containing the isomorphism class of $R$
consists of $26$ isomorphism classes.
\end{remark}
\begin{figure}
\def\hevA{0}
\def\he{4}
\setlength{\unitlength}{1.4mm}
{
\begin{picture}(80,10)(3, 0)
\put(16, \he){\circle{1}}
\put(15, \hevA){$\alpha$}
\put(22, \he){\circle{1}}
\put(20.5, \hevA){$\beta_1\spar{+}$}
\put(22.5, \he){\line(5, 0){5}}
\put(28, \he){\circle{1}}
\put(26.5, \hevA){$\beta_1\spar{-}$}
\put(34, \he){\circle{1}}
\put(32.5, \hevA){$\beta_2\spar{+}$}
\put(34.5, \he){\line(5, 0){5}}
\put(40, \he){\circle{1}}
\put(38.5, \hevA){$\beta_2\spar{-}$}
\put(50, \he){$\dots$}
\put(62, \he){\circle{1}}
\put(60.5, \hevA){$\beta_6\spar{+}$}
\put(62.5, \he){\line(5, 0){5}}
\put(68, \he){\circle{1}}
\put(66.5, \hevA){$\beta_6\spar{-}$}
\end{picture}
}
\caption{Basis of $R_0$}\label{fig:basisR0}
\end{figure}
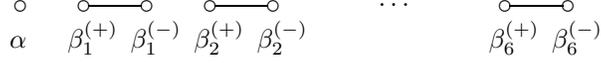
The image $\iota(\ampleX)\in \PPP_{26}\cap L_{26}$ of the ample class $\ampleX \in S_X$ defined by~\eqref{eq:theample} satisfies
\begin{equation}\label{eq:RootsR}
\Roots(([\iota (\ampleX )]\inj L_{26})\sperp) =\Roots((\iota \colon S_X \inj L_{26})\sperp) \cong \Roots(R),
\end{equation}
where $[\iota (\ampleX )]$ is the sublattice of $L_{26}$ generated by $\iota (\ampleX )$.
Hence $\ampleX $ is an interior point of an $\LS$-chamber,
which we denote by $\LSD_0$.
Moreover, we have
\[
\Sep_{26}( \iota(\ampleX ), \iota(\theh))=\emptyset,
\]
where we denote by  $\Sep_{26}$ the set of separating $(-2)$-vectors in $L_{26}$.
Hence the class $\theh$ is a point of  $\LSD_0$.
We choose a vector $\tilample\in  \PPP_L \cap L_{26}$ that satisfies  
\[
\Roots(([\tilample] \inj L_{26})\sperp) =\emptyset,
\quad 
 \Sep_{26}(\iota(\ampleX ), \tilample)=\emptyset.
\]
Then $\tilample$ is an interior point of a Conway chamber $\Concham_0$ such that 
$\iota\inv (\Concham_0)=\LSD_0$.
We can calculate a subset of the set of roots $\tilr$ of $L_{26}$ such that
 $\Concham_0 \cap (\tilr)\sperp$ is a wall of $\Concham_0$,
either by Vinberg's algorithm~\cite{VinbergBombay}, 
or by calculating $\Sep_{26}(\tilample, \vect{v})$,
where $\vect{v}\in \PPP_{26}\cap L_{26}$ are randomly chosen vectors.
If this subset  is  large enough, 
these roots $\tilr$ span $L_{26}\tensor\Q$ and hence
the Weyl vector $\weyl_0$ of the Conway chamber  $\Concham_0$ is 
calculated by solving 
the equations $\intf{\weyl_0, \tilr}=1$.
\begin{remark}
The $\ADE$-type  of 
the roots 
in~\eqref{eq:RootsR} is
 $A_1+6 A_2$.
Hence the hyperplanes perpendicular to these roots 
decompose $R\tensor \R$ into $2\times 6^6$ regions.
Therefore there exist exactly $2\times 6^6$ Conway chambers $\Concham$ such that 
$\iota\inv (\Concham)=\LSD_0$.
\end{remark}
Thus we  prepared all the data necessary to start the algorithm
of~\cite[Section 4.1]{BrandhorstShimada2021}
to calculate a complete set $V_0$ of the representatives of $V/G$ and a finite generating set of $G=\Aut(X)$.
We executed this algorithm.
The computation terminated
and yielded the following: 
\begin{proposition}\label{prop:seven}
The set $V_0$ consists of the following seven $\LS$-chambers:
\[
\LSD_0, \; \; \LSD_1\spar{1}, \; \LSD_1\spar{2},  \;  \LSD_1\spar{3},  \;  \LSD_1\spar{4},  \;  \LSD_1\spar{5},  \;  \LSD_1\spar{6}.
\]
\end{proposition}
We will describe each of these $\LS$-chambers in $V_0$,
and during the description, we  present automorphisms
in the set $\HHH$ defined by~\eqref{eq:HHH}.
\par
We use the following convention.
Let $\LSD$ be an $\LS$-chamber, and let $\Concham$ be a Conway chamber such that $\iota\inv(\Concham)=\LSD$.
Let $\weyl$ be the Weyl vector of $\Concham$.
For a wall $w$ of $\LSD$, 
let $v\in S_X\dual$ be the primitive defining vector of $w$ (see~Definition~\ref{def:primitive}), and
we put 
\[
n(w):=\intf{v, v}, \quad a(w):=\intf{\weyl, \iota(v)},   \quad h(w):=\intf{\theh, v}.
\]
These rational numbers are useful in classifying walls.
\subsection{The $\LS$-chamber $\LSD_0$}
The initial $\LS$-chamber $\LSD_0$ contains the ample class $\ampleX $ in its interior.
 The stabilizer subgroup of $\LSD_0$ in $G$ is $\{1, \invol(\theh)\}$.
The group $M$ leaves 
 $\LSD_0$ invariant.
The chamber $\LSD_0$  has $110$ walls, and 
the action of $M$ decomposes the walls of $\LSD_0$ into four orbits 
 $\orb_1,\orb_2, \orb_3, \orb_4$ of sizes
 $2$, $12$, $6$, $90$, respectively.
  The data of these orbits are given in Table~\ref{table:walls_of_C0}.
  \begin{table}
 \[
 \renewcommand{\arraystretch}{1.2}
 \begin{array}{c | c| c c c | l }
  & \textrm{size} & n & a & h & \\
  \hline 
  \orb_1  & 2 &-2 & 1 & 2 & \vgamma\spar{\pm} \\
  \orb_2  & 12 & -2 & 1 & 0 &\ve_{i}\spar{\pm}  \\
  \orb_3  & 6 &-3/2 & 3/2 & 1 & \text{isom with } \LSD_1\spar{\alpha}  \\
  \orb_4  & 90 &-2/3 & 3 & 2 & \text{isom with } \LSD_0
  \end{array}
 \]
 \vskip 5pt
 \caption{Walls of $\LSD_0$}\label{table:walls_of_C0}
 \end{table}
  \par
 The orbit $\orb_1$ of size $2$ consists of  
 $(\vgamma\spar{\pm})\sperp \cap \LSD_0$.
The orbit $\orb_2$ of size $12$ consists of 
 $(\ve_i\spar{\pm})\sperp \cap  \LSD_0$.
 Hence the $\LS$-chamber adjacent to $\LSD_0$ across a wall in $\orb_1$ or $\orb_2$
 is not  contained in $N_X$.
 \par
 The orbit $\orb_3$ of size $6$ consists of the walls $(v_{\alpha})\sperp\cap \LSD_0$
whose primitive defining vectors $v_{\alpha}$ are given by
 \begin{equation}\label{eq:wiofC0}
\intf{v_{\alpha}, \theh}=1,
\quad 
\intf{v_{\alpha}, \ve_i\spar{+}}=\intf{v_{\alpha}, \ve_i\spar{-}}=
\begin{cases}
1 & \textrm{if $i=\alpha$, }\\
0 & \textrm{if $i\ne \alpha$.}
\end{cases}
\end{equation}
%
Let $\LSD_1\spar{\alpha}$ be the $\LS$-chamber adjacent to $\LSD_0$ across the wall $(v_{\alpha})\sperp\cap \LSD_0$.
Then $\LSD_1\spar{\alpha}$ is contained in  $N_X$, but is not $G$-equivalent  
to $\LSD_0$, and any two of $\LSD_1\spar{1}, \dots, \LSD_1\spar{6}$ are not $G$-equivalent to each other.
Hence these chambers $\LSD_1\spar{\alpha}$ ($\alpha=1, \dots, 6$) are added to $V_0$ as new representatives of $V/G$.
 \par
The walls $w_{IJ}$ in the orbit $\orb_4$ of size $90$ are indexed by ordered pairs $(I, J)$,
where $I$ and $J$ are subsets of $\{1, \dots, 6\}$
satisfying $|I|=|J|=2$ and $I\cap J=\emptyset$.
The primitive defining vector $v_{IJ}\in S_X\dual$ of $w_{IJ}\in \orb_4$
is given by 
\begin{eqnarray*}
&& \intf{v_{IJ}, \theh}=2, \\
&&\intf{v_{IJ}, \ve_i\spar{+}}=0,\;\; \intf{v_{IJ}, \ve_i\spar{-}}=0, \qquad \textrm{if $i\notin I\cup J$, }\\
&&\intf{v_{IJ}, \ve_i\spar{+}}=1,\;\; \intf{v_{IJ}, \ve_i\spar{-}}=0, \qquad \textrm{if $i\in I$, }\\
&&\intf{v_{IJ}, \ve_i\spar{+}}=0,\;\; \intf{v_{IJ}, \ve_i\spar{-}}=1, \qquad \textrm{if $i\in J$.} 
\end{eqnarray*}
%
The $\LS$-chamber $\LSD_{IJ}$ adjacent to $\LSD_0$ across the wall $w_{IJ}$
is $G$-equivalent to $\LSD_0$.
An automorphism  $g_{IJ}\in G$ that maps $\LSD_0$ to $\LSD_{IJ}$  isomorphically  is given as follows.
Let $h_{IJ}$ be a vector of $S_X\tensor\Q$ defined by 
\begin{eqnarray}
&& \intf{h_{IJ}, \theh}=6, \nonumber\\
&&\intf{h_{IJ}, \ve_i\spar{+}}=0,\;\; \intf{h_{IJ}, \ve_i\spar{-}}=0, \qquad \textrm{if $i\notin I\cup J$, } \label{eq:hIJ}\\ 
&&\intf{h_{IJ}, \ve_i\spar{+}}=1,\;\; \intf{h_{IJ}, \ve_i\spar{-}}=1, \qquad \textrm{if $i\in I$, } \nonumber\\
&&\intf{h_{IJ}, \ve_i\spar{+}}=0,\;\; \intf{h_{IJ}, \ve_i\spar{-}}=3, \qquad \textrm{if $i\in J$.}  \nonumber
\end{eqnarray}
Then $h_{IJ} \in S_X$ and $\intf{h_{IJ}, h_{IJ}}=2$.
We confirm  $\Sep(h_{IJ}, \ampleX )=\emptyset$, and 
hence $h_{IJ}\in N_X$.
The complete linear system $|h_{IJ}|$ is proved to be fixed-component free by the criterion in~Section~\ref{subsec:findingaut}.
The involution $\invol(h_{IJ})$ associated with the double covering $\pi(h_{IJ})\colon X\to \P^2$
given by $|h_{IJ}|$ maps $\LSD_0$ to $\LSD_{IJ}$  isomorphically.
Therefore 
\[
\invol(h_{IJ})\inv=\invol(h_{IJ})\;\;\in\;\; \TG(\LSD_{IJ},  u_0(\LSD_{IJ}))
\]
in the notation of Section~\ref{subsec:VE}.
These involutions $\invol(h_{IJ})$  are the involutions of type (b) in Section~\ref{subsec:auttype}.
\begin{remark}\label{rem:PhiIJ}
Suppose that $I=\{i_1, i_2\}$, $J=\{j_1, j_2\}$, and
\[
\{1, \dots, 6\}-(I \cup J)=\{k_1, k_2\}.
\]
Then the  smooth rational curves on $X$ contracted to points by 
the double covering $\pi(h_{IJ})\colon X\to \P^2$ 
are as in Figure~\ref{fig:IJ},
where $\tilell_{j_1 j_2}$  is the curve given in Remark~\ref{rem:rats}.
In particular, the singular locus of the branch curve $B(h_{IJ})$
is of type $A_3+A_5$.
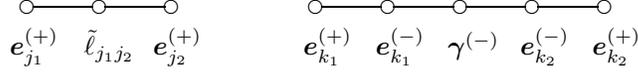
\begin{figure}
\def\hevA{0}
\def\he{4}
\setlength{\unitlength}{1.6mm}
{
\begin{picture}(80,10)(3, 0)
\put(16, \he){\circle{1}}
\put(14.6, \hevA){$\ve_{j_1}\spar{+}$}
\put(16.5, \he){\line(5, 0){5}}
\put(22, \he){\circle{1}}
\put(20.8, \hevA){$\tilell_{j_1 j_2}$}
\put(22.5, \he){\line(5, 0){5}}
\put(28, \he){\circle{1}}
\put(26.5, \hevA){$\ve_{j_2}\spar{+}$}
\put(40, \he){\circle{1}}
\put(39, \hevA){$\ve_{k_1}\spar{+}$}
\put(46, \he){\circle{1}}
\put(45, \hevA){$\ve_{k_1}\spar{-}$}
\put(52, \he){\circle{1}}
\put(51, \hevA){$\vgamma\spar{-}$}
\put(58, \he){\circle{1}}
\put(57, \hevA){$\ve_{k_2}\spar{-}$}
\put(64, \he){\circle{1}}
\put(63, \hevA){$\ve_{k_2}\spar{+}$}
\put(40.5, \he){\line(5, 0){5}}
\put(46.5, \he){\line(5, 0){5}}
\put(52.5, \he){\line(5, 0){5}}
\put(58.5, \he){\line(5, 0){5}}
\end{picture}
}
\caption{Exceptional curves of $\pi(h_{IJ})$}\label{fig:IJ}
\end{figure}
\end{remark}
\begin{remark}\label{rem:IJandJI}
We have ${v_{IJ}}^{\invol(\theh)}= v_{JI}$, ${h_{IJ}}^{\invol(\theh)} \ne h_{JI}$, and can confirm that  
the involution $\invol({h_{IJ}}^{\invol(\theh)} )= \invol(\theh)  \invol(h_{IJ}) \invol(\theh)$
is equal to $\invol(h_{JI})$.
\end{remark}
\subsection{The $\LS$-chamber $\LSD_1\spar{\alpha}$} 
The stabilizer subgroup of $\LSDaa$ in $G$ is $\{1, \invol(\theh)\}$.
The group $M$ acts on the set $\{\LSD_1\spar{1},  \dots, \LSD_1\spar{6}\}$ transitively.
Let $\Maa$ be the stabilizer subgroup of $\LSDaa$ in $M$.
Then $\Maa$ is isomorphic to $\Z/2\Z\times S_5$.
The chamber $\LSDaa$  has $110$ walls, and 
the action of $\Maa$ decomposes the walls of $\LSDaa$ into seven  orbits 
 $\orbb_1, \dots, \orbb_7$.
 The data of these orbits are given in Table~\ref{table:walls_of_LSDaa}.
  \begin{table}
 \[
 \renewcommand{\arraystretch}{1.2}
 \begin{array}{c | c| c c c | l }
  & \textrm{size} & n & a & h & \\
  \hline 
  \orbb_1  & 1 & -3/2 & 3/2 & -1 & \textrm{back to $\LSD_0$} \\
  \orbb_2  & 2 &-2 & 1 & 2 & \vgamma\spar{\pm} \\
  \orbb_3  & 5 &-2 & 1 & 2 & \tilell_{\alpha\beta} \;\; (\beta\ne \alpha) \\
  \orbb_4  & 10 &-2 & 1 & 0 & \ve_{\beta}\spar{\pm}  \;\; (\beta\ne \alpha)\\
  \orbb_5  & 2 & -3/2 &3/2 & 1 & \text{isom with } \LSDaa\\
  \orbb_6  & 30 & -1/6 &7/2 & 3 & \text{isom with } \LSD_1\spar{\beta}  \;\; (\beta\ne \alpha)\\
  \orbb_7  & 60 & -2/3 &3 & 2 & \text{isom with }  \LSD_1\spar{\beta}  \;\; (\beta\ne \alpha)\\
  \end{array}
 \]
 \vskip 5pt
 \caption{Walls of $\LSDaa$}\label{table:walls_of_LSDaa}
 \end{table}
  \par
  The orbit $\orbb_1$  consists of a single wall,
  and the adjacent $\LS$-chamber across this wall is $\LSD_0$,
which means that this wall is a wall in the orbit $\orb_3$ of walls of $\LSD_0$ viewed from the opposite side.
 \par
The orbit $\orbb_2$ of size $2$ consists of  
 $(\vgamma\spar{\pm})\sperp \cap  \LSDaa$,
the orbit $\orbb_3$ of size $5$ consists of 
 $(\tilell_{\alpha\beta})\sperp \cap  \LSDaa$ with  $\beta\ne\alpha$,
 and 
 the orbit $\orbb_4$ of size $10$ consists of 
 $(\ve_{\beta}\spar{\pm})\sperp \cap  \LSDaa$ with $\beta\ne\alpha$.
 The adjacent $\LS$-chambers across these walls are therefore  not contained in $N_X$.
 \par
The orbit $\orbb_5$ is of size $2$.
One of the walls in $\orbb_5$  is defined by a vector $\vaa^+\in S_X\dual$ satisfying 
\begin{eqnarray*}
&& \intf{\vaa^+, \theh}=1, \\
&&  \intf{\vaa^+, \ve_{\alpha}\spar{+}}=2, \;\;  \intf{\vaa^+, \ve_{\alpha}\spar{-}}=-1, \\
&&   \intf{\vaa^+, \ve_{\beta}\spar{+}}=0, \;\;  \intf{\vaa^+, \ve_{\beta}\spar{-}}=0\qquad(\beta\ne \alpha),
\end{eqnarray*}
and the other wall in $\orbb_5$ is defined by the vector 
\[
\vaa^-:=(\vaa^{+})^{\invol(\theh)}.
\]
The adjacent $\LS$-chamber $\LSD_{\alpha}^+$ across the wall $(\vaa^+)\sperp\cap\LSDaa$ is $G$-equivalent  to $\LSDaa$.
Indeed, the following 
 automorphism  $\invol(\haa^+)\in G$  maps $\LSDaa$ to $\LSD_{\alpha}^+$  isomorphically.
 Let $\haa^+$ be the vector defined by 
\begin{eqnarray}
&& \intf{\haa^+, \theh}=4, \nonumber\\
&&  \intf{\haa^+, \ve_{\alpha}\spar{+}}=2, \;\;  \intf{\haa^+, \ve_{\alpha}\spar{-}}=0,  \label{eq:haaplus}\\
&&   \intf{\haa^+, \ve_{\beta}\spar{+}}=0, \;\;  \intf{\haa^+, \ve_{\beta}\spar{-}}=1\;\;\;\; (\beta\ne\alpha). \nonumber
\end{eqnarray}
Then we have $\haa^+\in S_X$ and $\intf{\haa^+, \haa^+}=2$.
We confirm  $\Sep( \haa^+, \ampleX )=\emptyset$ and hence   $\haa^+\in N_X$.
The complete linear system $|\haa^+|$ is proved to be fixed-component free
by the criterion in Section~\ref{subsec:findingaut}.
Then we can confirm by direct computation that 
the involution $\invol(\haa^+)$ associated with the double covering $\pi(h^+_{\alpha})\colon X\to \P^2$
given by $|\haa^+|$ induces  $\LSDaa\cong \LSD_{\alpha}^+$.
It is obvious that 
the automorphism  $ \invol(\haa^-):=\invol(\theh) \invol(\haa^+)\invol(\theh)$ maps $\LSDaa$ to 
the adjacent $\LS$-chamber $\LSD_{\alpha}^-$ across the wall $(\vaa^-)\sperp\cap\LSDaa$.
Therefore we have 
\[
 \invol(\haa^{\pm})=\invol(\haa^{\pm}) \inv \in \TG(\LSD_{\alpha}^{\pm}, u_0(\LSD_{\alpha}^{\pm}))
\]
in the notation of Section~\ref{subsec:VE}.
These involutions  $\invol(\haa^{\pm})$ are the involutions of type (c) in Section~\ref{subsec:auttype}.
\begin{remark}\label{rem:typecexceptional}
The branch curve $B(h^+_{\alpha})$ of the double covering  $\pi(\haa^+)$ 
has the singularities of type $A_2+5A_1$.
The exceptional curves over the singular point of type $A_2$
are $\vgamma\spar{-}$ and $\ve_{\alpha}\spar{-}$,
whereas the exceptional curves over the singular points of type $A_1$
are $\ve_{\beta}\spar{+}$ for $\beta\ne\alpha$.
In particular, the involution $\invol(\haa^+)$ interchanges $\vgamma\spar{-}$ and $\ve_{\alpha}\spar{-}$.
 \end{remark}
 \par
 The description of the orbit $\orbb_6$
 is rather complicated, and hence is postponed to the next subsection.
  \par
 We describe the orbit $\orbb_7$ of size $60$.
Suppose that  $\beta\in \{1, \dots, 6 \}$ 
 and  $F=\{i_1, i_2\}\subset \{1, \dots, 6 \}$ satisfy $i_1\ne i_2$, 
 $\beta\ne \alpha$ and $\{\alpha, \beta\}\cap \{i_1, i_2\}=\emptyset$.
Let  $v_{\beta F}\spar{+}\in S_X\dual$ be the vector defined by    
\begin{eqnarray*}
&& \intf{v_{\beta F}\spar{+}, \theh}=2, \\
&&  \intf{v_{\beta F}\spar{+}, \ve_i\spar{+}}=1, \;\;  \intf{v_{\beta F}\spar{+}, \ve_i\spar{-}}=0 \;\;   \textrm{if $i\in \{\alpha, \beta\}$}, \\
&&  \intf{v_{\beta F}\spar{+}, \ve_i\spar{+}}=0, \;\;  \intf{v_{\beta F}\spar{+}, \ve_i\spar{-}}=1\;\;   \textrm{if $i\in F$}, \\
&&  \intf{v_{\beta F}\spar{+}, \ve_i\spar{+}}=0, \;\;  \intf{v_{\beta F}\spar{+}, \ve_i\spar{-}}=0\;\;  \textrm{otherwise}.
\end{eqnarray*}
We then put
\[
v_{\beta F}\spar{-}:=\left(v_{\beta F}\spar{+} \right)^{\invol(\theh)}.
\]
The orbit $\orbb_7$ consists of walls $(v_{\beta F}\spar{+})\sperp\cap \LSDaa$ and $(v_{\beta F}\spar{-})\sperp\cap \LSDaa$.
The adjacent $\LS$-chamber $\LSD_{ \beta  F}\spar{\pm}$ across the wall $(v_{\beta F}\spar{\pm})\sperp\cap \LSDaa$ 
is $G$-equivalent  to $\LSD_1\spar{\beta}$.
We put $A:=\{\alpha, \beta\}$,
and consider the polarization $h_{AF}$ of degree $2$ 
defined by~\eqref{eq:hIJ} with $I=A$ and $J=F$.
The involution $\invol(h_{AF})$, 
which is an involution of type (b) in Section~\ref{subsec:auttype},
maps  $\LSD_1\spar{\beta}$ to  $\LSD_{\beta F}\spar{+}$   isomorphically,
whereas the involution $\invol(h_{FA})$
maps  $\LSD_1\spar{\beta}$ to  $\LSD_{\beta F}\spar{-}$  isomorphically.
(See Remark~\ref{rem:IJandJI}.)
Therefore we have 
\[
\invol(h_{AF})=\invol(h_{AF}) \inv \in \TG(\LSD_{\beta F}\spar{+}, u_0(\LSD_{\beta F}\spar{+})),
\quad
\invol(h_{FA})=\invol(h_{FA}) \inv \in \TG(\LSD_{\beta F}\spar{-}, u_0(\LSD_{\beta F}\spar{-})),
\]
in the notation of Section~\ref{subsec:VE}.
\subsection{The orbit $\orbb_6$}\label{subsec:orbb6}
 %
%
In the following, 
for a sign $\sigma\in \{+, -\}$, let $\barsigma$ denote the opposite sign:  $\{\sigma, \barsigma\}=\{+, -\}$.
First,  
we define automorphisms $g\sprime_{\sigma I j}$ and $g\spprime_{\sigma J}$.
\par
Let $\III$ be the set of ordered triples 
\[
I=(\{i_1\}, \{i_2, i_3, i_4\}, \{i_5, i_6\})
\]
such that  $\{i_1, \dots, i_6\}=\{1, \dots, 6\}$.
We have $|\III|=60$.
For a pair of $\sigma\in \{+, -\}$ and $I\in \III$,
we have the configuration of smooth rational curves as in Figure~\ref{fig:ellD4A3}.
Then 
\begin{eqnarray*}
f_{\phi} :=f_{\sigma I}& :=& \ve_{i_1}\spar{\barsigma} + \ve_{i_2}\spar{\barsigma}+ \ve_{i_3}\spar{\barsigma} +\ve_{i_4}\spar{\barsigma}+ 2\vgamma\spar{\barsigma}\\
&=& \vgamma\spar{\sigma}+ \ve_{i_5}\spar{\sigma}+ \ve_{i_6}\spar{\sigma}+\tilell_{i_5 i_6}
\end{eqnarray*}
is the class of a fiber of an elliptic fibration $\phi\colon X\to \P^1$
with 
\[
z_{\phi} :=z_{\sigma I}:= \ve_{i_1}\spar{\sigma}
\]
 being the class of a section.
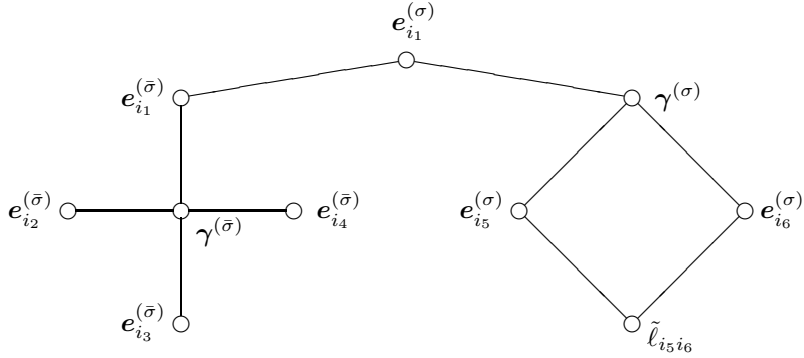
\begin{figure}
\newcommand{\node}{\circle{2}}
\setlength{\unitlength}{1mm}
{
\begin{picture}(80,56)(10, -5)
\put(5,20){\node}\put(-3,19){$\ve_{i_2}\spar{\barsigma}$}
\put(20,35){\node}\put(12,34){$\ve_{i_1}\spar{\barsigma}$}
\put(20,20){\node}\put(22,16){$\vgamma\spar{\barsigma}$}
\put(20,05){\node}\put(12,04){$\ve_{i_3}\spar{\barsigma}$}
\put(35,20){\node}\put(38,19){$\ve_{i_4}\spar{\barsigma}$}
\put(80,35){\node}\put(83,34){$\vgamma\spar{\sigma}$}
\put(65,20){\node}\put(57,19){$\ve_{i_5}\spar{\sigma}$}
\put(80,05){\node}\put(82,02){$\tilell_{i_5 i_6}$}
\put(95,20){\node}\put(97,19){$\ve_{i_6}\spar{\sigma}$}
\put(50,40){\node}\put(48,44){$\ve_{i_1}\spar{\sigma}$}
\put(20,6) {\line(0,1){13}}
\put(20,21) {\line(0,1){13}}
\put(19,20) {\line(-1,0){13}}
\put(21,20) {\line(1,0){13}}
\put(65.75,20.75) {\line(1,1){13.5}}
\put(65.75,19.25) {\line(1,-1){13.5}}
\put(94.25,20.75) {\line(-1,1){13.5}}
\put(94.25,19.25) {\line(-1,-1){13.5}}
\put(48.9,39.9) {\line(-6,-1){27.8}}
\put(51.1,39.9) {\line(6,-1){27.8}}
\end{picture} 
}
\caption{Configuration for a Jacobian fibration}\label{fig:ellD4A3}
\end{figure}
Thus we obtain 
a Jacobian fibration $\phi$ with the zero section $z_{\phi}$, and its Mordell-Weil group  
\[
\MW_{\phi}:=\MW(X, f_{\phi}, z_{\phi})\;\subset\; G=\Aut(X).
\]
Calculating the set $\Theta_{\phi}=\Roots(W_{\phi})\cap \Rats(X)$,
we see that 
the $\ADE$-type of the reducible fibers of $\phi\colon X\to \P^1$ is $D_4+A_3$.
Hence the rank of $\MW_{\phi}$ is $4$.
Since the trivial sublattice of $\phi$,
which is of rank $9$ generated by  the classes of the ten curves  in Figure~\ref{fig:ellD4A3},
is primitive in $S_X$, 
we see that $\MW_{\phi}$ is torsion free.
A Gram matrix of 
the Mordell-Weil lattice  $\MW_{\phi}$ (see Remark~\ref{rem:MWL}) is 
\[
\frac{3}{\;4\;}
\left[\begin{array}{cccc} 
3 & -1 & -1 & -1 \\ 
-1 & 3 & -1 & -1 \\ 
-1 & -1 & 3 & 1 \\ 
-1 & -1 & 1 & 3 
\end{array}\right].
\]
The numbers $n(s)$ of elements with small Mordell-Weil norms $s$ in $\MW_{\phi}$ are given as follows:
\[
\renewcommand{\arraystretch}{1.2}
\begin{array}{c|cccc}
s & 9/4 & 3 & 21/4 & 6 \\
\hline 
n(s) &12 & 14 & 16 & 30
\end{array}.
\]
Among these, we have the following sections of $\phi$:
\begin{itemize}
\item
The six smooth rational curves $\tilell_{j_1 j_2}$,
where $j_1\in \{i_2, i_3, i_4\}$ and $j_2\in \{i_5, i_6\}$,
satisfy $\intf{\tilell_{j_1 j_2}, f}=1$,
and hence they are sections of $\phi$.
Their Mordell-Weil norms are  $9/4$.
\item 
The three  smooth rational curves $\ve\spar{\sigma}_{j}$,
where $j\in \{i_2, i_3, i_4\}$,
also satisfy $\intf{\ve\spar{\sigma}_{j}, f}=1$,
and hence they are sections of $\phi$.
Their Mordell-Weil norms are equal to  $3$.
\end{itemize}
These $6+3$ sections $\tilell_{j_1 j_2}$ and $\ve\spar{\sigma}_{j}$  generate $\MW_{\phi}$.
\begin{definition}
For $j\in \{i_2, i_3, i_4\}$,
we denote by 
$g\sprime_{\sigma I j}$ the automorphism of $X$ obtained 
as the translation by the section $\ve\spar{\sigma}_{j}\in \MW_{\phi}$.
This is the automorphism of type (e) in Section~\ref{subsec:auttype}.
\end{definition}
Let $\JJJ$ be the set of ordered $4$-tuples 
\[
J=(\{i_1\}, \{i_2, i_3\}, \{i_4, i_5\}, \{i_6\})
\]
such that  $\{i_1, \dots, i_6\}=\{1, \dots, 6\}$.
We have $|\JJJ|=180$.
For a pair of $\sigma\in \{+, -\}$ and $J\in \JJJ$,
let $h_{\sigma J}$ be the vector of $S_X\tensor\Q$ defined by 
\begin{eqnarray}
&& \intf{h_{\sigma J}, \theh} = 14,  \nonumber\\
&&\intf{h_{\sigma J}, \ve_{i_1}\spar{\sigma}}=1\quand \intf{h_{\sigma J}, \ve_{i_1}\spar{\barsigma}}=0,\nonumber \\
&& \intf{h_{\sigma J}, \ve_{i}\spar{\sigma}}=4  \quand  \intf{h_{\sigma J}, \ve_{i}\spar{\barsigma}}=0  \qquad\textrm{for $i=i_2$ and $i= i_3$},\label{eq:hsigmaJ} \\
&& \intf{h_{\sigma J}, \ve_{i}\spar{\sigma}}=0  \quand \intf{h_{\sigma J}, \ve_{i}\spar{\barsigma}}=5  \qquad \textrm{for $i=i_4$ and $i= i_5$}, \nonumber  \\
&& \intf{h_{\sigma J}, \ve_{i_6}\spar{\sigma}}=5 \quand \intf{h_{\sigma J}, \ve_{i_6}\spar{\barsigma}}=4. \nonumber 
\end{eqnarray}
Then $h_{\sigma J} \in S_X$ and $\intf{h_{\sigma J}, h_{\sigma J}}=2$.
We confirm  $\Sep(h_{\sigma J}, \ampleX )=\emptyset$, and hence   $h_{\sigma J}\in N_X$.
The complete linear system $|h_{\sigma J}|$ is proved to be fixed-component free by the criterion in~Section~\ref{subsec:findingaut}.
\begin{definition}
We denote by 
$g\spprime_{\sigma J}$ the involution $\invol(h_{\sigma J})$.
This is the involution of type (d) in Section~\ref{subsec:auttype}.
\end{definition}
\begin{remark}
The  smooth rational curves on $X$ contracted by 
the double covering  $\pi(h_{\sigma J})\colon X\to \P^2$
associated with  $|h_{\sigma J }|$ are as in Figure~\ref{fig:sigmaJ}.
In particular, $\Sing(B(h_{\sigma J}))$ is of type $D_4+A_5$.
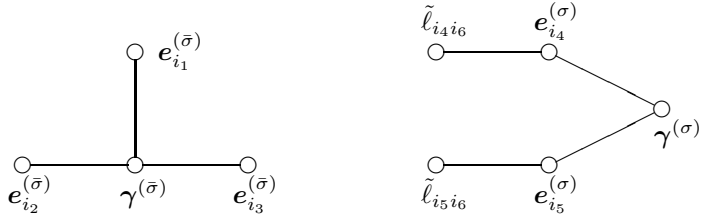
\begin{figure}
\newcommand{\node}{\circle{2}}
\setlength{\unitlength}{1mm}
{
\begin{picture}(90,40)(10, 5)
\put(5,15){\node}\put(3,10){$\ve_{i_2}\spar{\barsigma}$}
\put(20,15){\node}\put(18,10){$\vgamma\spar{\barsigma}$}
\put(20,30){\node}\put(23,29){$\ve_{i_1}\spar{\barsigma}$}
\put(35,15){\node}\put(33,10){$\ve_{i_3}\spar{\barsigma}$}
\put(60,15){\node}\put(58,10){$\tilell_{i_5 i_6}$}
\put(75,15){\node} \put(73,10){$\ve_{i_5}\spar{\sigma}$}
\put(60,30){\node}\put(58,33){$\tilell_{i_4 i_6}$}
\put(75,30){\node}\put(73,33){$\ve_{i_4}\spar{\sigma}$}
\put(90,22.5){\node}\put(89,18){$\vgamma\spar{\sigma}$}
\put(6,15) {\line(1,0){13}}
\put(21,15) {\line(1,0){13}}
\put(20,16) {\line(0,1){13}}
\put(61,15) {\line(1,0){13}}
\put(61,30) {\line(1,0){13}}
\put(75.9,29.6) {\line(2,-1){13.1}}
\put(75.9,15.4) {\line(2,1){13.1}}
\end{picture} 
}
\caption{Exceptional curves of $\pi(h_{\sigma J})$}\label{fig:sigmaJ}
\end{figure}
\end{remark}
We now describe the orbit $\orbb_6$ of walls of  $\LSDaa$.
The size of $\orbb_6$ is $30$.
Suppose that $\beta\in \{1, \dots, 6 \}$
 and $F=\{i_1, i_2\}\subset \{1, \dots, 6 \}$ satisfy $i_1\ne i_2$, 
 $\beta\ne \alpha$ and $\{\alpha, \beta\}\cap \{i_1, i_2\}=\emptyset$.
 Let  $u:=u_{\beta F}\in S_X\dual $ be the vector defined by    
\begin{eqnarray*}
&& \intf{u, \theh}=3, \\
&&  \intf{u, \ve_{\alpha}\spar{+}}=1, \;\;  \intf{u, \ve_{\alpha}\spar{-}}=1, \\
&&  \intf{u, \ve_{\beta}\spar{+}}=0, \;\;  \intf{u, \ve_{\beta}\spar{-}}=0,   \\ 
&&  \intf{u, \ve_i\spar{+}}=0, \;\;  \intf{u, \ve_i\spar{-}}=1\;\;   \textrm{if $i\in F$}, \\
&&   \intf{u, \ve_i\spar{+}}=1, \;\;  \intf{u, \ve_i\spar{-}}=0 \;\;  \textrm{if $i\notin \{\alpha, \beta\}\cup F$}.
\end{eqnarray*}
The orbit $\orbb_6$ consists of walls $(u_{\beta F})\sperp\cap \LSDaa$.
The  $\LS$-chamber $\LSD_{\alpha\beta F}$ adjacent to $\LSDaa$ across the wall $(u_{\beta F})\sperp\cap \LSDaa$ 
is $G$-equivalent  to $\LSD_1\spar{\beta}$.
An automorphism $g_{\alpha\beta F} \in G$ that maps $\LSD_1\spar{\beta}$ to $\LSD_{\alpha\beta F}$ isomorphically 
is given as follows.
We put
\[
K:=\{1, \dots, 6\}\; \setminus\;(\{\alpha, \beta\}\;\cup\; F).
\]
Then we have
\begin{equation}\label{eq:gabF}
g_{\alpha\beta F}=g\sprime_{+I\beta} \cdot g\spprime_{+J}=g\sprime_{-I\sprime\beta}\cdot g\spprime_{-J\sprime}, 
\end{equation}
where 
\begin{eqnarray*}
&& I=(\{\alpha\}, K\cup\{\beta\}, F)\in \III, \quad J=(\{\beta\}, K, F, \{\alpha\})\in \JJJ, \\
&& I\sprime=(\{\alpha\}, F\cup\{\beta\}, K)\in \III, \quad J\sprime=(\{\beta\}, F,K, \{\alpha\})\in \JJJ.
\end{eqnarray*}
Therefore we have 
\[
g_{\alpha\beta F} \inv \in \TG(\LSD_{\alpha\beta F}, u_0(\LSD_{\alpha\beta F}))
\]
in the notation of Section~\ref{subsec:VE}.
\begin{remark}\label{rem:trying}
The equality~\eqref{eq:gabF} was found 
by trying small combinations of the automorphisms of type (a)--(e).
\end{remark}
\subsection{Proof of Theorem~\ref{thm:genssimple}}
Any two distinct elements of $V_0$ are not $G$-equivalent.
Any $\LS$-chamber that is contained in $N_X$ and is adjacent to an element of $V_0$
is $G$-equivalent to an element of $V_0$.
Hence, by Proposition~\ref{prop:VE}, the set $V_0$ is a complete set of representatives of $V/G$.
\par
As the set $\HHH$ defined by~\eqref{eq:HHH},
we can take the set consisting of the identity element $1$, all involutions of type (b), (c), 
and the automorphisms  $g_{\alpha\beta F}\inv$, where $g_{\alpha\beta F}$ is given by~\eqref{eq:gabF}
and is a product of automorphisms of type (d) and (e).
The stabilizer subgroup $\Stab_G(D_0)$ of the initial element  $D_0\in V_0$ is $\{1, \invol(\theh)\}$.
Hence, by Proposition~\ref{prop:VE}, the group $G=\Aut(X)$ 
is generated by
 the automorphisms of type (a)-(e).
\qed
\begin{remark}
This generating set is very redundant.
\end{remark}
\subsection{Proof of Theorem~\ref{thm:rats}}
We prove that $G=\Aut(X)$ acts on $\Rats(X)$ transitively.
Let $r$ be an arbitrary element of $\Rats(X)$.
Since $r$ defines a wall of $N_X$,
there exists an $\LS$-chamber $D$ contained in $N_X$
such that $r$ defines a wall of $D$.
We have an automorphism $g\in G$
such that $D^g\in V_0$.
By the description of walls of the representative $\LS$-chambers in $V_0$,
we see that 
$r^g$ is one of the 
$12+2+15$ smooth rational curves $\ve_{\alpha}\spar{\pm}$, $\vgamma\spar{\pm}$,
and $\tilell_{ij}$.
The action of $\invol(\theh)$ gives  
$\ve_{\alpha}\spar{+}\leftrightarrow \ve_{\alpha}\spar{-} $ and $\vgamma\spar{+} \leftrightarrow \vgamma\spar{-}$.
By Remark~\ref{rem:PhiIJ}, the involution $\invol(h_{IJ})$ of type (b) interchanges 
$\ve_{j_1}\spar{+}$ and $\ve_{j_2}\spar{+}$
(see Figure~\ref{fig:IJ}).
By Remark~\ref{rem:typecexceptional}, 
the involution $\invol(h_{\alpha}^+)$ of type (c) 
interchanges $\vgamma\spar{-}$ and $\ve_{\alpha}\spar{-}$.
As was shown in Section~\ref{subsec:orbb6}, 
the elliptic fibration $X\to \P^1$  given by $f_{\sigma I}$ 
has sections $\ve_j\spar{\sigma}$ and $\tilell_{j_1 j_2}$,
and hence they belong to the same $G$-orbit.
Therefore these $12+2+15$ smooth rational curves are in the same $G$-orbit.
\qed
\bibliographystyle{plain}
\bibliography{myrefsAut}
%

\end{document}

%% file: mymacros.tex
\newcommand{\erase}[1]{}

\newtheorem{theorem}{Theorem}[section]
\newtheorem{lemma}[theorem]{Lemma}
\newtheorem{proposition}[theorem]{Proposition}

\newtheorem{_algorithm}[theorem]{Algorithm}

\newtheorem{_definition}[theorem]{Definition}
\newenvironment{definition}{\begin{_definition}\rm}{\end{_definition}}

\newtheorem{_propositiondefinition}[theorem]{Proposition-Definition}

\newtheorem{_remark}[theorem]{\it Remark}
\newenvironment{remark}{\begin{_remark}\rm}{\end{_remark}}

\newtheorem{_example}[theorem]{Example}
\newenvironment{example}{\begin{_example}\rm}{\end{_example}}

\newtheorem{_assumption}[theorem]{Assumption}

\newtheorem{_construction}[theorem]{Construction}

\newtheorem{_claim}[theorem]{Claim}

\newtheorem{_conjecture}[theorem]{Conjecture}

\newtheorem{_problem}[theorem]{Problem}

\numberwithin{equation}{section}
\numberwithin{table}{section}
\numberwithin{figure}{section}
\renewcommand{\qed}{\hfill {$\Box$}}

\newcommand{\C}{\mathord{\mathbb C}}

\renewcommand{\P}{\mathord{\mathbb  P}}
\newcommand{\Q}{\mathord{\mathbb  Q}}
\newcommand{\R}{\mathord{\mathbb R}}

\newcommand{\Z}{\mathord{\mathbb Z}}

\newcommand{\CCC}{\mathord{\mathcal C}}

\newcommand{\EEE}{\mathord{\mathcal E}}
\newcommand{\FFF}{\mathord{\mathcal F}}

\newcommand{\HHH}{\mathord{\mathcal H}}
\newcommand{\III}{\mathord{\mathcal I}}
\newcommand{\JJJ}{\mathord{\mathcal J}}

\newcommand{\LLL}{\mathord{\mathcal L}}
\newcommand{\MMM}{\mathord{\mathcal M}}

\newcommand{\OOO}{\mathord{\mathcal O}}
\newcommand{\PPP}{\mathord{\mathcal P}}
\newcommand{\QQQ}{\mathord{\mathcal Q}}

\newcommand{\maprightsp}[1]{\; \smash{\mathop{\; \longrightarrow \; }\limits\sp{#1}}\; }

\newcommand{\mapdown}{\phantom{\Big\downarrow}\hskip -8pt \downarrow}

\newcommand{\mapdownleft}[1]{\llap{$\vcenter{\hbox{$\scriptstyle#1$\;}}$}\mapdown}
\newcommand{\mapdownsurj}{
\hbox{$\bigm\downarrow$}
\llap{\hbox{\raise 2pt\hbox{$\bigm\downarrow$}}}%
\vstrechmapdown
}

\newcommand{\mapupsurj}{
\hbox{$\bigm\uparrow$}
\llap{\hbox{\raise 2pt\hbox{$\bigm\uparrow$}}}%
\vstrechmapup
}

\newcommand{\inj}{\hookrightarrow}
\newcommand{\surj}{\mathbin{\to \hskip -7pt \to}}

\newcommand{\set}[2]{\{\,{#1}\mid {#2} \,\}}

\newcommand{\gen}[1]{\langle {#1}  \rangle}

\newcommand{\tensor}{\otimes}

\newcommand{\sprime}{\sp\prime}

\newcommand{\spar}[1]{\sp{(#1)}}

\newcommand{\spprime}{\sp{\prime\prime}}
\newcommand{\spcirc}{\sp{\mathord{\circ}}}

\newcommand{\sperp}{\sp{\perp}}

\newcommand{\dual}{\sp{\vee}}

\newcommand{\semidirectproduct}{\rtimes}

\newcommand{\inv}{\sp{-1}}

\newcommand{\bdr}{\partial\,}

\newcommand{\GL}{\mathord{\mathrm{GL}}}

\newcommand{\OG}{\mathord{\mathrm{O}}}

\newcommand{\id}{\mathord{\mathrm{id}}}

\newcommand{\Image}{\operatorname{\mathrm{Im}}\nolimits}
\newcommand{\Aut}{\operatorname{\mathrm{Aut}}\nolimits}

\newcommand{\Sing}{\operatorname{\mathrm{Sing}}\nolimits}
\newcommand{\Spec}{\operatorname{\mathrm{Spec}}\nolimits}
\newcommand{\rank}{\operatorname{\mathrm{rank}}\nolimits}

\newcommand{\closure}[1]{\overline{#1}}

\newcommand{\rmand}{\textrm{and}}

\newcommand{\quand}{\quad\rmand\quad}

\newcommand{\mystrutd}[1]{\phantom{\hbox{\vrule depth #1}}}
\newcommand{\mystruthd}[2]{\phantom{\hbox{\vrule  height #1 depth #2}}}

\newcommand{\intf}[1]{\langle #1 \rangle}

\newcommand{\Stab}{\operatorname{\rm Stab}}

%% file: mymacrosAut.tex
\newcommand{\SX}{S_X}
\newcommand{\Roots}{\mathord{\mathrm{Roots}}}
\newcommand{\Rats}{\mathord{\mathrm{Rats}}}

\newcommand{\Sep}{\mathord{\mathrm{Sep}}}
\newcommand{\barPPP}{\overline{\PPP}}
\newcommand{\MW}{\mathrm{MW}}
\newcommand{\ADE}{\mathrm{ADE}}

\newcommand{\tilTheta}{\widetilde{\Theta}}

\newcommand{\specialization}{\mathrm{sp}}
\newcommand{\bargamma}{\overline{\gamma}}

\newcommand{\clN}{\overline{N}}
\newcommand{\barX}{\overline{X}}
\newcommand{\clX}{\overline{X}}
\newcommand{\barGamma}{\overline{\Gamma}}
\newcommand{\LS}{L_{26}/S_X}
\newcommand{\vect}[1]{\boldsymbol{#1}}
\newcommand{\theh}{\vect{h}}
\newcommand{\weyl}{\mathbf{w}}
\newcommand{\ve}{\vect{e}}
\newcommand{\ampleX}{\vect{a}}

\newcommand{\tilr}{\tilde{r}}
\newcommand{\orb}{\mathord{o}}

\newcommand{\vgamma}{\vect{\gamma}}
\newcommand{\tilample}{\tilde{\ampleX}}
\newcommand{\tilell}{\tilde{\ell}}

\newcommand{\Maa}{M_{\alpha}}

\newcommand{\Xfg}{X_{f, g}}

\newcommand{\TG}{T_G}
\newcommand{\orbb}{o\sprime}

\newcommand{\invol}{i}

\newcommand{\phaa}{\phantom{a}}

\newcommand{\zerosec}{\zeta}
\newcommand{\mwsec}{\sigma}

\newcommand{\addE}{{+_{E}}}

\newcommand{\haa}{h_{\alpha}}
\newcommand{\vaa}{v_{\alpha}}
\newcommand{\intfnull}{\intf{\phantom{a}\phantom{a}}}
\newcommand{\LSD}{D}
\newcommand{\LSDaa}{\LSD_1^{(\alpha)}}

\newcommand{\Card}{\mathrm{Card}}
\newcommand{\Concham}{\mathbf{C}}
\newcommand{\ConC}{\Concham}

\newcommand{\barsigma}{\bar{\sigma}}

\newcommand{\MWphi}{\MW_{\phi}}

%% file: figures.tex
\begin{figure}
\def\ha{40}
\def\hav{37}
\def\hd{25}
\def\hdv{22}
\def\he{10}
\def\hev{7}
\def\hevA{6.2}
\def\hevE{6.3}
\setlength{\unitlength}{1.4mm}
{\small
\begin{picture}(80,16)(-5, 6)
\put(34, 16){\circle*{1}}\put(34, 16){\circle{1.7}}
\put(36.5, 16.7){$C_{\nu, 0}$}
\put(13, \hevA){$C_{\nu, 1}$}
\put(20.0, \hevA){$C_{\nu, 2}$}
\put(44.5, \hevA){$C_{\nu, \ell-1}$}
\put(52.5, \hevA){$C_{\nu, \ell}$}
\put(16, \he){\circle{1}}\put(16, \he){\circle{1.7}}
\put(22, \he){\circle{1}}\put(22, \he){\circle{1.7}}
\put(28, \he){\circle{1}}\put(28, \he){\circle{1.7}}
\put(46, \he){\circle{1}}\put(46, \he){\circle{1.7}}
\put(52, \he){\circle{1}}\put(52, \he){\circle{1.7}}
\put(36, \he){$\dots$}
\put(16.83, 10.5){\line(3, 1){16.2}}
\put(51.15, 10.5){\line(-3, 1){16.2}}
\put(16.8, \he){\line(5, 0){4.3}}
\put(22.8, \he){\line(5, 0){4.3}}
\put(28.8, \he){\line(5, 0){4.3}}
\put(40.8, \he){\line(5, 0){4.3}}
\put(46.8, \he){\line(5, 0){4.3}}
\put(65,12){A fiber of type $A_{\ell}$}
\end{picture}
}
\vskip .5cm
{\small
\begin{picture}(80,14)(-5, 6)
\put(10, 13.5){$C_{\nu, 1}$}
\put(10, 5){$C_{\nu, 2}$}
\put(21.5, \hev){$C_{\nu, 3}$}
\put(28, \hev){$C_{\nu,4}$}
\put(56.5, 13.5){$C_{\nu, 0}$}
\put(45.2, \hev){$C_{\nu, \ell-1}$}
\put(56.5, 5){$C_{\nu, \ell}$}
\put(16, 14){\circle{1}}\put(16, 14){\circle{1.7}}
\put(16, 6){\circle{1}}\put(16, 6){\circle{1.7}}
\put(28, \he){\circle{1}}
\put(22, \he){\circle{1}}
\put(49, \he){\circle{1}}
\put(55, 6){\circle{1}}\put(55, 6){\circle{1.7}}
\put(55,14){\circle*{1}}\put(55,14){\circle{1.7}}
\put(16.8, 6.3){\line(3,2){4.85}}
\put(16.8, 13.7){\line(3,-2){4.85}}
\put(49.42,10.4){\line(3,2){4.75}}
\put(49.42,9.6){\line(3,-2){4.75}}
\put(22.5, \he){\line(5, 0){5}}
\put(28.5, \he){\line(5, 0){5}}
\put(43.5, \he){\line(5, 0){5}}
\put(37.5, \he){$\dots$}
\put(65,\he){A fiber of type $D_{\ell}$}
\end{picture}
}
\vskip .5cm
{\small
\begin{picture}(55,20)(-5, 6)
\put(23.7, 21.5){$C_{\nu, 0}$}
\put(23.7, 15.2){$C_{\nu, 1}$}
\put(7.8, \hevE){$C_{\nu, 2}$}
\put(14.5, \hevE){$C_{\nu, 3}$}
\put(21.0,\hevE){$C_{\nu, 4}$}
\put(27.5, \hevE){$C_{\nu, 5}$}
\put(34, \hevE){$C_{\nu, 6}$}
\put(22, 22){\circle*{1}}\put(22, 22){\circle{1.7}}
\put(22, 16){\circle{1}}
\put(10, \he){\circle{1}}\put(10, \he){\circle{1.7}}
\put(16, \he){\circle{1}}
\put(22, \he){\circle{1}}
\put(28, \he){\circle{1}}
\put(34, \he){\circle{1}}\put(34, \he){\circle{1.7}}
\put(22, 10.5){\line(0,1){5}}
\put(22, 16.5){\line(0,1){4.6}}
\put(10.85, \he){\line(5, 0){4.6}}
\put(16.5, \he){\line(5, 0){5}}
\put(22.5, \he){\line(5, 0){5}}
\put(28.5, \he){\line(5, 0){4.6}}
\put(45,14){A fiber of type $E_6$}
\end{picture}
}
\vskip .5cm
{\small
\begin{picture}(67,14)(-5, 6)
\put(28, 16){\circle{1}}
\put(29.5, 15.5){$C_{\nu, 1}$}
\put(28, 10.5){\line(0,1){5}}
\put(8.0, \hevE){$C_{\nu, 0}$}
\put(14.5, \hevE){$C_{\nu, 2}$}
\put(20.8,  \hevE){$C_{\nu, 3}$}
\put(27.1,  \hevE){$C_{\nu, 4}$}
\put(33.4,  \hevE){$C_{\nu, 5}$}
\put(39.4,  \hevE){$C_{\nu, 6}$}
\put(45.5,  \hevE){$C_{\nu, 7}$}
\put(10, \he){\circle*{1}}\put(10, \he){\circle{1.7}}
\put(16, \he){\circle{1}}
\put(22, \he){\circle{1}}
\put(28, \he){\circle{1}}
\put(34, \he){\circle{1}}
\put(40, \he){\circle{1}}
\put(46, \he){\circle{1}}\put(46, \he){\circle{1.7}}
\put(10.8, \he){\line(5, 0){4.7}}
\put(16.5, \he){\line(5, 0){5}}
\put(22.5, \he){\line(5, 0){5}}
\put(28.5, \he){\line(5, 0){5}}
\put(34.5, \he){\line(5, 0){5}}
\put(40.5, \he){\line(5, 0){4.6}}
\put(55,14){A fiber of type $E_7$}
\end{picture}
}
\vskip .5cm
{\small
\begin{picture}(70,14)(-5, 6)
\put(23.5, 15.5){$C_{\nu, 1}$}
\put(22, 10.5){\line(0,1){5}}
\put(7.8, \hevE){$C_{\nu, 2}$}
\put(14.6,  \hevE){$C_{\nu, 3}$}
\put(20.6,  \hevE){$C_{\nu, 4}$}
\put(27.1,  \hevE){$C_{\nu, 5}$}
\put(33.3,  \hevE){$C_{\nu, 6}$}
\put(39.4,  \hevE){$C_{\nu, 7}$}
\put(45.5,  \hevE){$C_{\nu, 8}$}
\put(51.5,  \hevE){$C_{\nu, 0}$}
\put(22, 16){\circle{1}}
\put(10, \he){\circle{1}}
\put(16, \he){\circle{1}}
\put(22, \he){\circle{1}}
\put(28, \he){\circle{1}}
\put(34, \he){\circle{1}}
\put(40, \he){\circle{1}}
\put(46, \he){\circle{1}}
\put(52, \he){\circle*{1}}\put(52, \he){\circle{1.7}}
\put(10.5, \he){\line(5, 0){5}}
\put(16.5, \he){\line(5, 0){5}}
\put(22.5, \he){\line(5, 0){5}}
\put(28.5, \he){\line(5, 0){5}}
\put(34.5, \he){\line(5, 0){5}}
\put(40.5, \he){\line(5, 0){5}}
\put(46.5, \he){\line(5, 0){4.6}}
\put(60,14){A fiber of type $E_8$}
\end{picture}
}
\vskip .3cm
\parbox{10cm}{
$C_{\nu, 0}$ is indicated by 
\begin{picture}(4, 6)(0, 0)\put(1.5,0.5){\circle*{1}}\put(1.5,0.5){\circle{1.7}}\end{picture}, and \\
$C_{\nu, j}$ for $j \in J_{\nu}-\{0\}$ is indicated by   \begin{picture}(4, 3)(0, 0)\put(1.5,0.5){\circle{1}}\put(1.5,0.5){\circle{1.7}}\end{picture}.
}
\vskip .2cm
\caption{Reducible fibers}\label{fig:affDynkin}
\end{figure}

%% file: AutXfgShimadaRevised2.bbl
\begin{thebibliography}{10}

\bibitem{Artal1994}
Enrique Artal-Bartolo.
\newblock Sur les couples de {Z}ariski.
\newblock {\em J. Algebraic Geom.}, 3(2):223--247, 1994.

\bibitem{ACT2008}
Enrique Artal~Bartolo, Jos\'{e}~Ignacio Cogolludo, and Hiro-o Tokunaga.
\newblock A survey on {Z}ariski pairs.
\newblock In {\em Algebraic geometry in {E}ast {A}sia---{H}anoi 2005},
  volume~50 of {\em Adv. Stud. Pure Math.}, pages 1--100. Math. Soc. Japan,
  Tokyo, 2008.

\bibitem{Artin1966}
Michael Artin.
\newblock On isolated rational singularities of surfaces.
\newblock {\em Amer. J. Math.}, 88:129--136, 1966.

\bibitem{BHPV2004}
Wolf~P. Barth, Klaus Hulek, Chris A.~M. Peters, and Antonius Van~de Ven.
\newblock {\em Compact complex surfaces}, volume~4 of {\em Ergebnisse der
  Mathematik und ihrer Grenzgebiete. 3. Folge. A Series of Modern Surveys in
  Mathematics [Results in Mathematics and Related Areas. 3rd Series. A Series
  of Modern Surveys in Mathematics]}.
\newblock Springer-Verlag, Berlin, second edition, 2004.

\bibitem{Borcherds1}
Richard Borcherds.
\newblock Automorphism groups of {L}orentzian lattices.
\newblock {\em J. Algebra}, 111(1):133--153, 1987.

\bibitem{Borcherds2}
Richard~E. Borcherds.
\newblock Coxeter groups, {L}orentzian lattices, and {$K3$} surfaces.
\newblock {\em Internat. Math. Res. Notices}, (19):1011--1031, 1998.

\bibitem{BrandhorstShimada2021}
Simon Brandhorst and Ichiro Shimada.
\newblock Automorphism groups of certain {E}nriques surfaces.
\newblock {\em Found. Comput. Math.}, Published online, 2021.
\newblock doi: 10.1007/s10208-021-09530-y.

\bibitem{Conway1983}
J.~H. Conway.
\newblock The automorphism group of the {$26$}-dimensional even unimodular
  {L}orentzian lattice.
\newblock {\em J. Algebra}, 80(1):159--163, 1983.

\bibitem{Degtyarev2008}
Alex Degtyarev.
\newblock On deformations of singular plane sextics.
\newblock {\em J. Algebraic Geom.}, 17(1):101--135, 2008.

\bibitem{Ebeling2013}
Wolfgang Ebeling.
\newblock {\em Lattices and codes}.
\newblock Advanced Lectures in Mathematics. Springer Spektrum, Wiesbaden, third
  edition, 2013.
\newblock A course partially based on lectures by Friedrich Hirzebruch.

\bibitem{GAP}
The~GAP Group.
\newblock {G}{A}{P} - {G}roups, {A}lgorithms, and {P}rogramming.
\newblock Version 4.11.0 of 29-Feb-2020 (http://www.gap-system.org).

\bibitem{Horikawa1975}
Eiji Horikawa.
\newblock On deformations of quintic surfaces.
\newblock {\em Invent. Math.}, 31(1):43--85, 1975.

\bibitem{KKS2014}
Toshiyuki Katsura, Shigeyuki Kondo, and Ichiro Shimada.
\newblock On the supersingular {$K3$} surface in characteristic 5 with {A}rtin
  invariant 1.
\newblock {\em Michigan Math. J.}, 63(4):803--844, 2014.

\bibitem{Kodaira1963}
K.~Kodaira.
\newblock On compact analytic surfaces. {II}, {III}.
\newblock {\em Ann. of Math. (2) 77 (1963), 563--626; ibid.}, pages 1--40,
  1963.
\newblock Reprinted in Collected Works, vol. III, pp. 1269--1372, Iwanami and
  Princton University Press 1975.

\bibitem{Kondo1998}
Shigeyuki Kondo.
\newblock The automorphism group of a generic {J}acobian {K}ummer surface.
\newblock {\em J. Algebraic Geom.}, 7(3):589--609, 1998.

\bibitem{KondoShimada2014b}
Shigeyuki Kondo and Ichiro Shimada.
\newblock On a certain duality of {N}\'{e}ron-{S}everi lattices of
  supersingular {$K3$} surfaces.
\newblock {\em Algebr. Geom.}, 1(3):311--333, 2014.

\bibitem{LieblichMaulik2018}
Max Lieblich and Davesh Maulik.
\newblock A note on the cone conjecture for {K}3 surfaces in positive
  characteristic.
\newblock {\em Math. Res. Lett.}, 25(6):1879--1891, 2018.

\bibitem{Looijenga2014}
Eduard Looijenga.
\newblock Discrete automorphism groups of convex cones of finite type.
\newblock {\em Compos. Math.}, 150(11):1939--1962, 2014.

\bibitem{Neron1964}
Andr\'{e} N\'{e}ron.
\newblock Mod\`eles minimaux des vari\'{e}t\'{e}s ab\'{e}liennes sur les corps
  locaux et globaux.
\newblock {\em Inst. Hautes \'{E}tudes Sci. Publ. Math.}, (21):128, 1964.

\bibitem{Nikulin1979}
V.~V. Nikulin.
\newblock Integer symmetric bilinear forms and some of their geometric
  applications.
\newblock {\em Izv. Akad. Nauk SSSR Ser. Mat.}, 43(1):111--177, 238, 1979.
\newblock English translation: Math USSR-Izv. 14 (1979), no. 1, 103--167
  (1980).

\bibitem{Nikulin1990}
Viacheslav~V. Nikulin.
\newblock Weil linear systems on singular {$K3$} surfaces.
\newblock In {\em Algebraic geometry and analytic geometry ({T}okyo, 1990)},
  ICM-90 Satell. Conf. Proc., pages 138--164. Springer, Tokyo, 1991.

\bibitem{OkaPho2002}
Mutsuo Oka and Duc~Tai Pho.
\newblock Classification of sextics of torus type.
\newblock {\em Tokyo J. Math.}, 25(2):399--433, 2002.

\bibitem{Torelli1971}
I.~I. Pjatecki{\u\i}-{\v{S}}apiro and I.~R. {\v{S}}afarevi{\v{c}}.
\newblock Torelli's theorem for algebraic surfaces of type {${\rm K}3$}.
\newblock {\em Izv. Akad. Nauk SSSR Ser. Mat.}, 35:530--572, 1971.
\newblock Reprinted in I. R. Shafarevich, Collected Mathematical Papers,
  Springer-Verlag, Berlin, 1989, pp. 516–557.

\bibitem{SaintDonat1974}
B.~Saint-Donat.
\newblock Projective models of {$K-3$} surfaces.
\newblock {\em Amer. J. Math.}, 96:602--639, 1974.

\bibitem{Schuett2016}
Matthias Sch\"{u}tt.
\newblock Dynamics on supersingular {K}3 surfaces.
\newblock {\em Comment. Math. Helv.}, 91(4):705--719, 2016.

\bibitem{MWLbook}
Matthias Sch\"{u}tt and Tetsuji Shioda.
\newblock {\em Mordell-{W}eil lattices}, volume~70 of {\em Ergebnisse der
  Mathematik und ihrer Grenzgebiete. 3. Folge. A Series of Modern Surveys in
  Mathematics [Results in Mathematics and Related Areas. 3rd Series. A Series
  of Modern Surveys in Mathematics]}.
\newblock Springer, Singapore, 2019.

\bibitem{Shimada2010}
Ichiro Shimada.
\newblock Lattice {Z}ariski {$k$}-ples of plane sextic curves and
  {$Z$}-splitting curves for double plane sextics.
\newblock {\em Michigan Math. J.}, 59(3):621--665, 2010.

\bibitem{Shimada2014}
Ichiro Shimada.
\newblock Projective models of the supersingular {$K3$} surface with {A}rtin
  invariant 1 in characteristic 5.
\newblock {\em J. Algebra}, 403:273--299, 2014.

\bibitem{Shimada2015}
Ichiro Shimada.
\newblock An algorithm to compute automorphism groups of {$K3$} surfaces and an
  application to singular {$K3$} surfaces.
\newblock {\em Int. Math. Res. Not. IMRN}, (22):11961--12014, 2015.

\bibitem{Shimada2016}
Ichiro Shimada.
\newblock Automorphisms of supersingular {$K3$} surfaces and {S}alem
  polynomials.
\newblock {\em Exp. Math.}, 25(4):389--398, 2016.

\bibitem{compdata}
Ichiro Shimada.
\newblock Explanation of the computation data for the paper ``{M}ordell-{W}eil
  groups and automorphism groups of elliptic {$K3$} surfaces", 2022.\\
\newblock
  \url{https://www.math.sci.hiroshima-u.ac.jp/~shimada/K3andEnriques.html}.

\bibitem{ShiodaMWL}
Tetsuji Shioda.
\newblock On the {M}ordell-{W}eil lattices.
\newblock {\em Comment. Math. Univ. St. Paul.}, 39(2):211--240, 1990.

\bibitem{Sterk1985}
Hans Sterk.
\newblock Finiteness results for algebraic {$K3$} surfaces.
\newblock {\em Math. Z.}, 189(4):507--513, 1985.

\bibitem{Tate1975}
J.~Tate.
\newblock Algorithm for determining the type of a singular fiber in an elliptic
  pencil.
\newblock In {\em Modular functions of one variable, {IV} ({P}roc. {I}nternat.
  {S}ummer {S}chool, {U}niv. {A}ntwerp, {A}ntwerp, 1972)}, volume Vol. 476 of
  {\em Lecture Notes in Math}, pages pp 33--52. Springer, Berlin, 1975.

\bibitem{VinbergBombay}
\`E.~B. Vinberg.
\newblock Some arithmetical discrete groups in {L}oba\v{c}evski\u{\i} spaces.
\newblock In {\em Discrete subgroups of {L}ie groups and applications to moduli
  ({I}nternat. {C}olloq., {B}ombay, 1973)}, pages 323--348. 1975.

\bibitem{VinbergSurvey}
\`E.~B. Vinberg.
\newblock Hyperbolic groups of reflections.
\newblock {\em Uspekhi Mat. Nauk}, 40(1(241)):29--66, 255, 1985.
\newblock English translation: Russian Math. Surveys 40 (1985), no. 1, 31–75.

\bibitem{Zariski1931}
Oscar Zariski.
\newblock On the irregularity of cyclic multiple planes.
\newblock {\em Ann. of Math. (2)}, 32(3):485--511, 1931.

\end{thebibliography}
